\documentclass[a4paper]{amsart} %

\usepackage[T1]{fontenc} 
\usepackage[utf8]{inputenc}
\usepackage[british]{babel}
\usepackage{newpxtext} %
\usepackage{eulervm} %

\usepackage{amssymb,textcomp,mathrsfs,stmaryrd,braket,commath,relsize}
\usepackage{bm} %
	\SetSymbolFont{stmry}{bold}{U}{stmry}{m}{n} %
\usepackage[new]{old-arrows} %

\usepackage{mathtools}
    \mathtoolsset{showonlyrefs} %

\theoremstyle{plain} %
    \newtheorem{theorem}{Theorem}[section]
    \newtheorem*{theorem*}{Theorem}
    \newtheorem{proposition}{Proposition}[section]
    \newtheorem*{proposition*}{Proposition}
    \newtheorem{corollary}{Corollary}[section]
    \newtheorem*{corollary*}{Corollary}
    \newtheorem{lemma}{Lemma}[section]
    \newtheorem*{lemma*}{Lemma}
    
    \newtheorem*{conjecture*}{Conjecture}
    
\theoremstyle{definition} %
    \newtheorem{definition}{Definition}[section]
    \newtheorem*{definition*}{Definition}

\theoremstyle{remark} %
    \newtheorem{remark}{Remark}[section]
    \newtheorem*{remark*}{Remark}
	\newtheorem{example}{Example}[section]
	\newtheorem*{example*}{Example}

\usepackage{etoolbox} 
    \newcommand{\addQEDstyle}[2]{\AtBeginEnvironment{#1}{\pushQED{\qed}\renewcommand{\qedsymbol}{#2}}
    \AtEndEnvironment{#1}{\popQED}} %
    \addQEDstyle{remark}{$\triangle$}
    \addQEDstyle{remark*}{$\triangle$} %
	\addQEDstyle{example}{$\triangle$}
    \addQEDstyle{example*}{$\triangle$} %
	
\usepackage{tikz}
	\usetikzlibrary{cd} %
	\usetikzlibrary{calc}
	\tikzstyle{vertex}=[circle,draw,inner sep=0pt,minimum size=12pt]

\usepackage{longtable}

\AtBeginDocument{\def\MR#1{}} %
\apptocmd{\sloppy}{\hbadness 10000\relax}{}{} %

\makeatletter %
	\@namedef{subjclassname@2020}{
		\textup{2020} Mathematics Subject Classification}
\makeatother

\usepackage{mymacros}
\newcommand\lthra{\relbar\joinrel\twoheadrightarrow}

\usepackage[colorlinks,psdextra,pdfencoding=auto]{hyperref}
    \hypersetup{colorlinks,linkcolor={red!50!black},filecolor={green!50!black},urlcolor={blue!80!black},citecolor={blue!50!black}}

\begin{document}

\title[Topology of irregular isomonodromy times]{Topology of irregular isomonodromy times on a fixed pointed curve} %

\author[J.~Douçot]{Jean Douçot} %
\thanks{During this project, J.~D. was funded by FCiências.ID.}

\author[G.~Rembado]{Gabriele Rembado}
\thanks{During this project, G.~R. was supported by the Deutsche Forschungsgemeinschaft (DFG, German Research Foundation) under Germany’s Excellence Strategy - GZ 2047/1, Projekt-ID 390685813.}

\address[J.~Douçot]{Section de mathématiques, Université de Genève, Rue du Conseil-Général 7-9, 1205 Genève (Switzerland)}
\curraddr{Institut de Recherche Mathématique Avancée,
	7 rue René Descartes,
	67084 Strasbourg Cedex,
	France
}
\email{jean.doucot@math.unistra.fr}

\address[G.~Rembado]{Hausdorff Centre for Mathematics, University of Bonn, 60 Endenicher Allee, D-53115 Bonn (Germany)}
\curraddr{Departament de Matemàtiques,
	Universitat Politècnica de Catalunya (UPC),
	Campus Diagonal Sud,
	Av. Doctor Marañón 44-50,
	08028 Barcelona,
	Spain}
\email{gabriele.rembado@upc.edu}

\date{\today}



\begin{abstract}
	We will define and study (moduli) spaces of deformations of irregular classes on Riemann surfaces, which provide an intrinsic viewpoint on the `times' of irregular isomonodromy systems in general.
	Our aim is to study the deeper generalisation of the $G$-braid groups that occur as fundamental groups of such deformation spaces, with particular focus on the generalisation of the full $G$-braid groups.
\end{abstract}

{\let\newpage\relax\maketitle} %

\setcounter{tocdepth}{1}  %
\tableofcontents

\section*{Introduction}

\renewcommand{\thetheorem}{\arabic{theorem}} %

Classically, the theory of isomonodromy constitutes a collection of nonlinear integrable differential equations, whose unknown is a (linear) meromorphic connection on a vector bundle over the Riemann sphere.
Geometrically, these are flat Ehresmann connections on a bundle whose fibres are moduli spaces of such meromorphic connections.

The underlying deformation parameters, the `times', have recently been given an intrinsic formulation, leading to a generalisation of the moduli of pointed curves (in any genus).
This framework is especially useful when considering the generalised deformations, beyond the generic case: recall~\cite{jimbo_miwa_ueno_1981_monodromy_preserving_deformation_of_linear_ordinary_differential_equations_with_rational_coefficients_i_general_theory_and_tau_function} set up a theory of `generic' isomonodromic deformations of meromorphic connections on vector bundles over a Riemann surface $\Sigma$, where the leading coefficient at each pole has distinct eigenvalues (building on~\cite{birkhoff_1913_the_generalized_riemann_problem_for_linear_differential_equations_and_the_allied_problems_for_linear_difference_and_q_difference_equations}; cf.~\cite{malgrange_1983_sur_les_deformations_isomonodromiques_ii_singularites_irregulieres,balser_jurkat_lutz_1979_birkhoff_invariants_and_stokes_multipliers_for_meromorphic_linear_differential_equations}).
This has been extended in two directions: i) replacing vector bundles by principal $G$-bundles, leading in particular to the appearance of $G$-braid groups for complex reductive groups $G$~\cite{boalch_2002_g_bundles_isomonodromy_and_quantum_weyl_groups};
and ii) considering \emph{nongeneric} admissible deformations (e.g.,
as in~\cite{boalch_2012_simply_laced_isomonodromy_systems,boalch_2014_geometry_and_braiding_of_stokes_data_fission_and_wild_character_varieties}), where the (untwisted/unramified) irregular type of the connection is \emph{arbitrary}, leading to cabled braid groups~\cite{doucot_rembado_tamiozzo_2026_local_wild_mapping_class_groups_and_cabled_braids}.

In particular the spaces of generalised monodromy data, the wild character varieties (a.k.a.~wild Betti spaces), have been proved to form a local system of Poisson varieties~\cite{boalch_2014_geometry_and_braiding_of_stokes_data_fission_and_wild_character_varieties}\fn{
	Basically speaking, a bundle of Poisson manifolds equipped with a complete flat connection: the (Betti) isomonodromy connection, a.k.a.~the wild nonabelian Gauß--Manin connection.}
\begin{equation}
	\label{eq:betti_bundle}
	\ul{\mc M}_{\on B} \lra \bm B \, ,
\end{equation}
over any space $\bm B$ of admissible deformations.
These give a purely topological description of the nonlinear isomonodromy differential equations, via the Riemann--Hilbert--Birkhoff correspondence.

\vspace{5pt}

Our purpose in this paper is to study the fundamental groups of the base spaces $\bm B$ of such admissible deformations, the groups that will act by algebraic Poisson automorphisms on the wild character varieties (the fibres of~\eqref{eq:betti_bundle}) from the parallel transport of the isomonodromy connection---i.e., the monodromy of the nonlinear differential equations.
This builds on our previous paper~\cite{doucot_rembado_tamiozzo_2026_local_wild_mapping_class_groups_and_cabled_braids}, which used a fixed marking: here we will quotient by the Weyl group action and get to the full version of `wild' mapping class groups, in analogy to forgetting the ordering of marked points on the underlying pointed curve.

\vspace{5pt}

This encompasses the much-studied case of regular singular connections, involving the complex character varieties, which is the entry point for the standard mapping-class- and braid-group-actions in classical/quantum 2$d$ gauge theories---via deformations of pointed curves, e.g.~\cite{kohno_1987_monodromy_representations_of_braid_groups_and_yang_baxter_equations,drinfeld_1989_quasi_hopf_algebras_and_knizhnik_zamolodchikov_equations,masbaum_2003_quantum_representations_of_mapping_class_groups,andersen_2006_asymptotic_faithfulness_of_the_quantum_su_n_representations_of_the_mapping_class_groups} in the quantum case.
The case of poles of order 2, however, has been extensively studied by various authors: there are relations to quantum groups  \cite{boalch_2002_g_bundles_isomonodromy_and_quantum_weyl_groups,xu_2020_stokes_phenomenon_and_yang_baxter_equations,xu_representations_of_quantum_groups_arising_from_the_stokes_phenomenon}, and already there the boundary of the space of times has a rich structure (corresponding to the `coalescence' of irregular times, cf.~\cite{cotti_dubrovin_guzzetti_2019_isomonodromy_deformations_at_an_irregular_singularity_with_coalescing_eigenvalues,xu_regularized_limits_of_stokes_matrices_isomonodromy_deformation_and_crystal_basis,xu_2023_on_a_connection_formula_of_a_higher_rank_analog_of_painleve_vi}).
In particular the simplest irregular singular case has been understood in rigorous analytic way, while this paper focuses on the algebro-geometric aspects of the general nongeneric case.

\vspace{5pt}

In this series of `local' papers we fix the underlying pointed curve, and vary the rest of the wild Riemann surface structure~\cite{boalch_2014_geometry_and_braiding_of_stokes_data_fission_and_wild_character_varieties}, i.e., the irregular types/classes, controlling principal parts of irregular singular connections beyond their (formal) residues.
More precisely~\cite{doucot_rembado_tamiozzo_2026_local_wild_mapping_class_groups_and_cabled_braids} studies the (fine) moduli space of untwisted (a.k.a.~unramified) irregular types for the split Lie algebra $(\mf g,\mf t) \ceqq \bigl( \Lie(G),\Lie(T) \bigr)$, where $T \sse G$ is a maximal torus, while here we consider irregular classes.

Recall in brief that an untwisted irregular type $Q$ at a point $a \in \Sigma$ is the germ of a $\mf t$-valued meromorphic function based there, defined up to holomorphic terms:
\begin{equation}
	\label{eq:irregular_type_intro}
	Q = \sum_{j = 1}^p A_j z^{-j} \in \mf t (\!(z)\!) \bs \mf t \llb z \rrb \, , \qquad A_j \in \mf t \, ,
\end{equation}
in a local coordinate $z$ vanishing at the marked point.
Then the Weyl group $W_{\mf g} = N_G(T) \bs T$ acts on the left tensor factor of
\begin{equation*}
	\mf t \ots_{\mb C} \bigl( \mb C (\!(z)\!) \bs \mb C \llb z \rrb \bigr) \simeq \mf t (\!(z)\!) \bs \mf t \llb z \rrb \, ,
\end{equation*}
and the irregular class underlying~\eqref{eq:irregular_type_intro} is its projection $\ol Q$ in the quotient, i.e., the Weyl-orbit through $Q$~\cite[Rmk.~10.6]{boalch_2014_geometry_and_braiding_of_stokes_data_fission_and_wild_character_varieties}.\fn{
	E.g. if $G = \GL_n(\mb C)$ we thus consider the coefficients $A_j$ in~\eqref{eq:irregular_type} up to simultaneous permutations of their eigenvalues.}

The important fact is the fibres of~\eqref{eq:betti_bundle} only depend on the collection of irregular classes underlying the irregular types at each marked point, and thus any (admissible) space of irregular classes provides an intrinsic topological description of the corresponding isomonodromy times.
In the generic case, where the leading coefficient of~\eqref{eq:irregular_type_intro} is out of all root hyperplanes, the homotopy type of the deformation space brings about the $G$-braid group: in this paper we shall encounter a generalisation in the nongeneric case, which we relate to braid cabling in type $A$.

\subsection*{Layout of the paper and main results}

In \S~\ref{sec:local_wmcg} we give the main definition: to a one-pointed (bare) wild Riemann surface $\bm{\Sigma} = (\Sigma,a,\ol Q)$ we associate a full/nonpure local `wild' mapping class group $\Gamma_{\ol Q}$ (WMCG), viz. the fundamental group of a space $\bm B_{\ol Q}$ of admissible deformations of the irregular class $\ol Q$ (cf.~Def.~\ref{def:local_WMCG}).
The latter is a topological quotient of the (universal) admissible deformation space $\bm B_Q$ of $Q$, where $Q$ is any irregular type lifting $\ol Q$.

\vspace{5pt}

In \S~\ref{sec:weyl_fission} we describe the subgroup of $W_{\mf g}$ preserving $\bm B_Q \sse \mf t^p$, and further the quotient thereof that acts freely; the resulting subquotient is denoted by $W_{\mf g \mid \bm{\mf h}}$.
The relevant statements are proven inductively along the sequence of fission/Levi (root) subsystems of $\Phi_{\mf g}$ associated with $Q$ (cf.~\cite{doucot_rembado_tamiozzo_2026_local_wild_mapping_class_groups_and_cabled_braids}): first in the case where $Q = Az^{-1}$ has a single coefficient $A \in \mf t$ (in \S~\ref{sec:weyl_fission_step}), and then in the general case (in \S~\ref{sec:weyl_fission_general}).

\begin{theorem*}[Cf.~Thm.~\ref{thm:local_wmcg_extension}]

	The space $\bm B_Q$ is a Galois covering of $\bm B_{\ol Q}$ with $\Gal(\bm B_Q,\bm B_{\ol Q}) = W_{\mf g \mid \bm{\mf h}}$, so $\Gamma_{\ol Q}$ is an extension of $W_{\mf g \mid \bm{\mf h}}$ by the pure local WMCG.
\end{theorem*}

\vspace{5pt}

In \S~\ref{sec:nonpure_low_rank_examples} we describe all full/nonpure local WMCGs for the irreducible rank-2 root systems, after explaining the classification boils down to simple Lie algebras.

\vspace{5pt}

Finally in \S~\ref{sec:nonpure_type_A_example} we explicitly describe the full/nonpure local WMCGs when $\mf g \in \Set{\mf{gl}_n(\mb C), \mf{sl}_n(\mb C)}$, in the nonabelian case $n \geq 2$.
This means identifying the `effective' subquotient of the Weyl group that controls the Galois covering (a Coxeter-type group), and then compute the fundamental group of the base (an Artin-type group).
The inductive step $Q = Az^{-1}$ is in \S~\ref{sec:nonpure_type_A_example_step}, where we prove the following statement.

\begin{theorem*}[Cf.~Prop.~\ref{prop:kernel_stabiliser_type_A}, Cor.~\ref{cor:reduced_weyl_group_type_A} and Prop.~\ref{prop:semipure_braid_group}]
	The Weyl-stabiliser of $\bm B_Q$ is a direct product of wreath products of symmetric groups, and the effective quotient $W_{\mf g \mid \bm{\mf h}}$ is a direct product of symmetric groups; then $\Gamma_{\ol Q}$ is the subgroup of braids whose underlying permutation lies in $W_{\mf g \mid \bm{\mf h}}$, and it is an extension of the latter by a pure braid group.
\end{theorem*}

In the general case instead we introduce a family of trees $(T,\bm r)$ with some decoration, called `ranked' fission trees, which depend on the choice of the irregular class $\ol Q$ (cf.~Def.~\ref{def:ranked_fission_tree}, and compare with the unranked fission trees of~\cite[\S~5]{doucot_rembado_tamiozzo_2026_local_wild_mapping_class_groups_and_cabled_braids}).
Their automorphisms control the Coxeter-type groups in the general type-$A$ case:

\begin{theorem*}[Cf.~Thm.~\ref{thm:tree_automorphisms} and Prop.~\ref{prop:split_sequence_weyl_fission_type_A}]
	The automorphism group $\Aut(T,\bm r)$ of the ranked fission tree is isomorphic to $W_{\mf g \mid \bm{\mf h}}$.
\end{theorem*}

Finally we attach a (full/nonpure) `cabled' braid group $\ms B(T,\bm r)$ to any ranked fission tree, in Def.~\ref{def:cabled_braid_group}, with a recursive algorithm (along maximal subtrees): this relies on the operadic composition of the symmetric and braid group operads, extending the \emph{pure} cabled braid group of~\cite{doucot_rembado_tamiozzo_2026_local_wild_mapping_class_groups_and_cabled_braids}---which rests in turn on the pure braid group operad.

The main result of \S~\ref{sec:nonpure_type_A_example} is that the elements of type-$A$ full/nonpure local WMCGs are precisely such `cabled' braids, in a proof of the multilevel braiding conjecture that appears e.g. in~\cite{ramis_2012_iso_irregular_deformations_of_linear_ode_and_dynamics_of_painleve_equations}, extending~\cite{doucot_rembado_tamiozzo_2026_local_wild_mapping_class_groups_and_cabled_braids}:
this is an instance of the new braiding first envisioned in~\cite{boalch_2014_geometry_and_braiding_of_stokes_data_fission_and_wild_character_varieties}.

\begin{theorem*}[Cf.~Thm.~\ref{thm:type_A_local_wmcgs_are_cabled_braid_groups}]
	There is a group isomorphism $\Gamma_{\ol Q} \simeq \ms B(T,\bm r)$, where $(T,\bm r)$ is the ranked fission tree associated with the irregular class $\ol Q$.
\end{theorem*}

All the Lie algebras in this text are tacitly defined over $\mb C$, and all tensor products are $\mb C$-bilinear.

Some basic notions and conventions, used throughout the body of the paper, are collected in \S~\ref{sec:app_notions}, while \S~\ref{sec:app_lemmata} contains the proofs of few lemmata---%
which have been postponed,
aiming to ease the reading flow.
In \S~\ref{sec:app_isomonodromy_systems} we spell out the relation of wild Riemann surfaces with the much-studied Hamiltonian viewpoint on isomonodromic deformations.

The end of remarks/examples is signaled by a $\triangle$.

\renewcommand{\thetheorem}{\arabic{section}.\arabic{theorem}} %

\section{Full/nonpure local WMCG}
\label{sec:local_wmcg}

Let $\Sigma$ be a closed Riemann surface, $G$ a connected complex reductive Lie group, $\mf g = \Lie(G)$ its Lie algebra, $T \sse G$ a maximal torus, and $\mf t = \Lie(T) \sse \mf g$ the associated Cartan subalgebra.
Denote then by $\Phi_\mf g = \Phi(\mf g,\mf t) \sse \mf t^{\dual}$ the root system of the split Lie algebra $(\mf g, \mf t)$, and by $W_\mf g = N_G(T) \bs T$ the Weyl group.

Choose a point $a \in \Sigma$, and let
\begin{equation}
	\label{eq:irregular_type}
	Q \in \mf t \ots \ms T_{\Sigma,a} \, , \qquad \ms T_{\Sigma,a} \ceqq \wh{\ms K}_{\Sigma,a} \bs \wh{\ms O}_{\Sigma,a} \, ,
\end{equation}
be an untwisted irregular type based there, introducing the completed local ring $\wh{\ms O}_{\Sigma,a}$ of $(\Sigma,a)$ and its fraction field $\wh{\ms K}_{\Sigma,a}$.
Recall that if $z$ is a local coordinate on $\Sigma$ with $z(a) = 0$ then~\eqref{eq:irregular_type} becomes
\begin{equation*}
	Q = \sum_{i = 1}^p A_i z^{-i} \in z^{-1} \mf t[z^{-1}] \simeq \mf t (\!(z)\!) \bs \mf t \llb z \rrb \, ,
\end{equation*}
for suitable coefficients $A_i \in \mf t$ and for an integer $p \geq 1$.
(Hereafter we refer to this simply as an `irregular type', always untwisted/unramified.)

As explained in the introduction, the---%
unframed---%
moduli spaces attached to~\eqref{eq:irregular_type} (the de Rham/Betti spaces~\cite{boalch_hyperkaehler_manifolds_and_nonabelian_hodge_theory_of_irregular_curves}) only depend on the Weyl-orbit of~\eqref{eq:irregular_type}, denoted by
\begin{equation}
	\label{eq:bare_irregular_type}
	\ol Q \in \bigl( \mf t \ots \ms T_{\Sigma,a} \bigr) \bs W_\mf g \, .
\end{equation}
Here the Weyl group acts on the Cartan subalgebra,
and trivially on the rightmost tensor factor; the element~\eqref{eq:bare_irregular_type} defines an irregular class, a.k.a.~a~`bare' irregular type~\cite[Rmk.~10.6]{boalch_2014_geometry_and_braiding_of_stokes_data_fission_and_wild_character_varieties} (cf.~\cite{boalch_yamakawa_twisted_wild_character_varieties,boalch_doucot_rembado_2025_twisted_local_wild_mapping_class_groups_configuration_spaces_fission_trees_and_complex_braids,doucot_rembado_yamakawa_twisted_local_g_wild_mapping_class_groups} in the twisted case).

If $Q$ is a `starting' irregular type, then we have associated with it the (universal) admissible deformation space $\bm B_Q$ in~\cite{doucot_rembado_tamiozzo_2026_local_wild_mapping_class_groups_and_cabled_braids}, cf.~\cite[Exmp.~10.1]{boalch_2014_geometry_and_braiding_of_stokes_data_fission_and_wild_character_varieties}.
In brief, this is the (connected) complex manifold $\bm B_Q = \prod_{i = 1}^p \bm B_{A_i}$, with
\begin{equation}
	\label{eq:deformation_irregular_type_factor}
	\bm B_{A_i} \ceqq \bigcap_{d_{\alpha} < i} \Ker(\alpha) \cap \bigcap_{d_{\alpha} = i} \bigl( \mf t \sm \Ker(\alpha) \bigr) \sse \mf t \, ,
\end{equation}
where
\begin{equation}
	\label{eq:pole_orders}
	d_{\alpha} = \ord (\alpha \circ Q) \, , \qquad \alpha \in \Phi_\mf g \, ,
\end{equation}
taking the pole order at $a \in \Sigma$.

Then the pure local wild mapping class group of the wild Riemann surface $(\Sigma,a,Q)$ is $\Gamma_Q \ceqq \pi_1(\bm B_Q,Q)$, i.e., the fundamental group of the (pointed) deformation space of $Q$.

\begin{remark}[Terminology]
	\label{rem:terminology}
	The term `pure' is reminiscent of pure braid groups, which are fundamental groups of configuration spaces of \emph{ordered} points.
	Below we will instead consider the `full/nonpure' case, which in turn is an analogue of the fundamental group for \emph{unordered} configurations.
	The latter configuration space arises by modding out the action of a symmetric group (permuting the points): then in our situation this is generalised by an action of the Weyl group.

	The picture in the case of simple poles is as follows.
	Let $\mc O \sse \mf g$ be a semisimple adjoint $G$-orbit, so that $\mc O \cap \mf t \sse \mc O$ is nonempty: define a \emph{marking} of $\mc O$ to be the choice of a point $A \in \mc O \cap \mf t$.
	In turn the intersection $\mc O \cap \mf t \sse \mf t$ then coincides with the $W_{\mf g}$-orbit of the vector $A \in \mf t$, and the Weyl group acts transitively (and generically freely) by forgetting the choice of marking.
	This is essentially because semisimple orbits correspond bijectively to the quotient $\mf t \bs W_{\mf g}$, cf.~\cite[\S~2]{collingwood_mcgovern_1993_nilpotent_orbits_in_semisimple_lie_algebras}.

	Then the topology of each semisimple orbit is determined by the centraliser of its marked point, i.e., the Levi factor of a parabolic subgroup of $G$, and this yields a partition of $\mf t \bs W_{\mf g}$.
	In particular there is a bulk consisting of $\mf t_{\reg} \bs W_{\mf g}$, i.e., the generic/regular semisimple elements, whose fundamental group is the full/nonpure $\mf g$-braid group (cf.~\S~\ref{sec:app_notions}).
	In the nongeneric case one instead finds more complicated hyperplane arrangements (possibly noncrystallographic~\cite{doucot_rembado_tamiozzo_2026_local_wild_mapping_class_groups_and_cabled_braids,rembado_2024_a_colourful_classification_of_quasi_root_systems_and_hyperplane_arrangements}), and more complicated reflection groups; so in turn more complicated fundamental groups.

	Finally we take this one step further, to get to the theory of wild/irregular singularities, considering more than semisimple orbits inside $\mf g$.
	Rather we consider `very good' orbits inside the (nonreductive) Lie algebra $\mf g_p = \mf g \llb z \rrb \bs z^p \mf g \llb z \rrb$, for an integer $p \geq 1$ as above, cf.~\cite[\S~1.5]{boalch_2018_wild_character_varieties_meromorphic_hitchin_systems_and_dynkin_diagrams}.
	Choosing an irregular type means fixing a \emph{marking} of such orbits, and here we get rid of this choice to get the intrinsic spaces of (local) isomonodromic deformations, whose fundamental group is then the `full/nonpure' (local) wild mapping class group.
\end{remark}

Concretely, two deformations of $Q$ are equivalent if they lie in the same $W_\mf g$-orbit inside $\mf t \ots \ms T_{\Sigma,a}$, in which case they define the same irregular class.
This leads to admissible deformations of the `starting' irregular class $\ol Q$, and to the main definition.

\begin{definition}
	\label{def:local_WMCG}

	The \emph{full/nonpure local wild mapping class group} of the (bare) wild Riemann surface $\bm \Sigma = (\Sigma,a,\ol Q \bigr)$ is
	\begin{equation}
		\label{eq:nonpure_local_WMCG}
		\Gamma_{\ol Q} \ceqq \pi_1 \bigl( \bm B_{\ol Q},\ol Q \bigr) \, ,
	\end{equation}
	where $\bm B_{\ol Q} = \bm B_Q \bs \!\!\sim$ is the topological quotient with respect to the equivalence relation
	\begin{equation}
		\label{eq:equivalence_relation}
		Q_1 \sim Q_2 \qquad \text{if} \qquad W_\mf g Q_1 = W_\mf g Q_2 \sse \mf t \ots \ms T_{\Sigma,a} \, .
	\end{equation}
\end{definition}

(Hereafter we will simply refer to~\eqref{eq:nonpure_local_WMCG} as the `WMCG', always full/nonpure and local.)

Note that $\Gamma_{\ol Q}$ depends on the root system $\Phi_{\mf g} \sse \mf t^{\dual}$, and on the tuple of integers $\bm d = (d_{\alpha})_{\alpha \in \Phi_\mf g}$.
It is in general larger then its pure counterpart, as some nonclosed paths in $\bm B_Q$ may become loops in $\bm B_{\ol Q}$.

\begin{remark}
	The space $\bm B_Q$ itself depends on the irregular type $Q$, not just on the underlying irregular class.
	However if $w \in W_\mf g$ then $\bm B_Q$ is homeomorphic to $w(\bm B_Q) = \bm B_{w(Q)}$, and~\eqref{eq:nonpure_local_WMCG} only depends on $\ol Q$.
\end{remark}

\begin{remark}[Intrinsic definition]
	Note one also has
	\begin{equation*}
		\bm B_Q = \bigcap_{\alpha \in \Phi_{\mf g}} \Set{ Q' = \sum_{i = 1}^p A'_i z^{-i} | \ord( \alpha \circ Q' ) = d_{\alpha} } \sse \mf t \ots \ms T_{\Sigma,a} \, ,
	\end{equation*}
	using the notation of~\eqref{eq:pole_orders}, and the factorisation~\eqref{eq:deformation_irregular_type_factor} follows (cf.~\cite[\S~1]{doucot_rembado_tamiozzo_2026_local_wild_mapping_class_groups_and_cabled_braids}).
	This gives another viewpoint on how the (germs of) meromorphic functions $q_{\alpha} = \alpha \circ Q$ determine the space.\fn{
		These are exponential factors of local fundamental solutions for the linear differential equations associated with the meromorphic connection with irregular type $Q$, featuring in the Stokes data of the connection.}

	In turn, the pole orders of the $q_{\alpha}$ are well defined up to local biholomorphisms of $\Sigma$ which fix the marked point $a \in \Sigma$.
	Hence the integers $d_{\alpha}$ only depend on the element $Q \in \mf t \ots \ms T_{\Sigma,a}$, and not on the identifications
	\begin{equation*}
		\wh{\ms O}_{\Sigma,a} \simeq \mb C \llb z \rrb \, , \quad \wh{\ms K}_{\Sigma,a} \simeq \mb C (\!(z)\!) \, , \quad \ms T_{\Sigma,a} \simeq \mb C (\!(z)\!) \bs \mb C \llb z \rrb \, .
	\end{equation*}
	In particular the deformation spaces $\bm B_Q$ and $\bm B_{\ol Q}$ are independent of a choice of local coordinate vanishing at the marked point.
\end{remark}

\begin{remark}
	Moreover all Cartan subalgebras $\mf t \sse \mf g$ are conjugated by (inner) Lie-algebra automorphisms of $\mf g$, which in turn induce homeomorphisms of the resulting deformation spaces $\bm B_Q$.
	Hence the isomorphism class of $\Gamma_Q$ does \textit{not} depend on the choice of the Cartan subalgebra, and so in turn neither does that of $\Gamma_{\ol Q}$.
\end{remark}

Now the Weyl action does \emph{not} preserve $\bm B_Q$ in the nongeneric case, i.e., the case where $A_p$ is not regular, so we first need to describe the subset
\begin{equation*}
	W_\mf g Q \cap \bm B_Q \sse W_\mf g Q \, ,
\end{equation*}
and further understand the Weyl-stabiliser of the irregular type.
(This is already visible in the tame case $p = 1$, cf.~Rmk.~\ref{rem:terminology} above.)

\section{Weyl group fission}
\label{sec:weyl_fission}

\subsection{Inductive step}
\label{sec:weyl_fission_step}

We first consider the case of a single coefficient, i.e., $Q = Az^{-1}$.
In the general case the irregular type is transformed along the diagonal Weyl action on each coefficient.

Choose then $A \in \mf t$, and let $\mf h = \mf Z_\mf g(A) \sse \mf g$ be the centraliser: it is the (reductive) Levi factor of a parabolic subalgebra of $\mf g$.
The associated deformation space~\eqref{eq:deformation_irregular_type_factor} becomes
\begin{equation}
	\label{eq:true_complement}
	\bm B_A = \Ker(\Phi_\mf h) \cap \bigcap_{\Phi_\mf g \sm \Phi_\mf h} \bigl( \mf t \sm \Ker(\alpha) \bigr) \sse \Ker(\Phi_\mf h) \, ,
\end{equation}
where
\begin{equation*}
	\Ker(\Phi_{\mf h}) = \Set{ A \in \mf t | \Braket{ \alpha | A } = 0 \text{ for } \alpha \in \Phi_{\mf h} } = \bigcap_{\Phi_{\mf h}} \Ker(\alpha) \sse \mf t \, .
\end{equation*}
For later use we set $U \ceqq \Ker(\Phi_\mf h)$.

Now if $w \in \Stab_{W_\mf g}(\bm B_A) \sse W_\mf g$ then certainly $wA \in \bm B_A$, but the converse is true.
To state this let $\mc O_A \subseteq \bm B_A$ be the orbit of $A$ under the action of $\Stab_{W_\mf g}(\bm B_A)$; then:
\begin{lemma}
	\label{lem:setwise_weyl_stabiliser}

	One has $\Stab_{W_\mf g}(U) = \Stab_{W_\mf g}(\bm B_A)$, and $(W_\mf g A) \cap \bm B_A = \mc O_A$.
\end{lemma}

\begin{proof}
	The Weyl group permutes the root hyperplanes via
	\begin{equation*}
		w \bigl( \Ker(\alpha) \bigr) = \Ker(w\alpha) \, , \qquad w \in W_\mf g \, , \, \alpha \in \Phi_\mf g \, ,
	\end{equation*}
	i.e., along the permutation of the roots. (Recall we identify $W(\Phi_{\mf{g}}) \sse \GL(\mf t^{\dual})$ and $W(\Phi_{\mf{g}}^{\dual}) \sse \GL(\mf t)$, cf.~\S~\ref{sec:app_notions}.)
	Hence $w(\bm B_A) \sse \bm B_A$ if and only if $w \in W_\mf g$ preserves the partition $\Phi_\mf g = \Phi_\mf h \cup (\Phi_\mf g \sm \Phi_\mf h)$, by~\eqref{eq:true_complement}.
	In turn this is equivalent to $w(\Phi_\mf h) \sse \Phi_\mf h$, proving the first statement.

	Analogously if $A' \ceqq wA \in \bm B_A$ then $A'$ lies on all the root hyperplanes of $\Phi_{\mf h}$, and out of all the root hyperplanes of the complement:
	it follows that $w$ preserves the above partition, whence the inclusion $(W_\mf g A) \cap \bm B_A \sse \mc O_A$---and the opposite one is tautological.
\end{proof}

Thus the restriction of orbits to the deformation space is controlled by the setwise stabiliser of $U \sse \mf t$.

\begin{remark}
	The extremal cases are $A = 0$, in which case $U = \mf t$ and $\Stab_{W_\mf g}(U) = 1$; and $A \in \mf t_{\reg}$, in which case $U = (0)$ and $\Stab_{W_\mf g}(U) = W_\mf g$.
\end{remark}

Now the Weyl group $W_\mf h = W(\Phi_\mf h) \sse W_\mf g$ of the Levi factor lies in the setwise stabiliser of $U$, but in general the inclusion is proper.
Namely by definition
\begin{equation*}
	\Ker(\alpha) = \Ker( \sigma_{\alpha} - 1 ) \sse \mf t \, , \qquad \alpha \in \Phi_\mf g \, ,
\end{equation*}
and the subgroup $W_\mf h$ is generated by the reflections along the hyperplanes of the subsystem $\Phi_\mf h \sse \Phi_\mf g$: hence automatically any element of $W_\mf h$ acts trivially on $U = \Ker(\Phi_\mf h)$, i.e.
\begin{equation*}
	W_\mf h \sse (W_\mf g)_U \sse \Stab_{W_\mf g}(U) \, .
\end{equation*}

However it is possible to show the first inclusion is an equality, and more precisely that $W_\mf h$ is the (maximal) parabolic subgroup fixing any element of~\eqref{eq:true_complement}:

\begin{lemma}
	\label{lem:parabolic_weyl_group}

	One has $W_\mf h = (W_\mf g)_U = (W_\mf g)_{A'}$, for all $A' \in \bm B_A$.
\end{lemma}

\begin{proof}
	In principle
	\begin{equation*}
		(W_\mf g)_U \sse (W_{\mf g})_{\bm B_A} \sse (W_\mf g)_{A'} \, , \qquad A' \in \bm B_A \, ,
	\end{equation*}
	so it is enough to show the inclusion $(W_\mf g)_{A'} \sse W_\mf h$, i.e., that any element of $W_\mf g$ fixing $A'$ lies in the Weyl group of $\Phi_\mf h$.

	To this end recall the Lie-group-theoretic definition of the Weyl group is
	\begin{equation*}
		W_\mf h = N_H(T) \bs T \sse N_G(T) \bs T = W_\mf g \, ,
	\end{equation*}
	using the normalisers $N_H(T) \sse N_G(T)$ of the given maximal torus, where $H \sse G$ is the (connected, reductive) Lie subgroup integrating $\mf h \sse \mf g$.
	But the centraliser of $A' \in \mf t$ is determined by the root hyperplanes to which $A'$ belongs, and so it coincides with $\mf h$.
	Thus an element $w \in W_\mf g$ such that $w(A') = A'$ corresponds to an element $g \in N_G(T)$---defined up to the $T$-action---such that $\Ad_g(A') = A'$: this means
	\begin{equation*}
		g \in N_G(T) \cap H = N_H(T) \, ,
	\end{equation*}
	whence $w \in W_\mf h$.
\end{proof}

Finally we have an identification $\mc O_A \simeq W_{\mf g \mid \mf h} A$, introducing the quotient group
\begin{equation}
	\label{eq:weyl_group_fission}
	W_{\mf g \mid \mf h} \ceqq \Stab_{W_\mf g}(U) \bs W_\mf h \, .
\end{equation}
Denote then by $\bm B_{\ol A}$ the topological quotient (of $\bm B_A$) of Def.~\ref{def:local_WMCG}, in the present case where $Q = Az^{-1}$.

\begin{proposition}
	\label{prop:nonpure_WMCG_factor}

	The fundamental group $\pi_1(\bm B_{\ol A},\mc O_A)$ is an extension of~\eqref{eq:weyl_group_fission} by $\pi_1(\bm B_A,A)$.
\end{proposition}

\begin{proof}
	The above discussion yields a homeomorphisms
	\begin{equation}
		\label{eq:unordered_configuration_space_factor}
		\bm B_{\ol A} \simeq \bm B_A \bs W_{\mf g \mid \mf h} \, ,
	\end{equation}
	by which we identify $\mc O_A$ as an element of $\bm B_{\ol A}$.

	Now by construction the $W_{\mf g \mid \mf h}$-action on $\bm B_A$ is free, in view of Lem.~\ref{lem:parabolic_weyl_group}.
	Moreover the action is automatically properly discontinuous ($W_{\mf g \mid \mf h}$ is finite), and the spaces involved are Hausdorff; hence every point of $\bm B_A$ has a neighbourhood $O \sse \bm B_A$ such that $\ol w_1(O) \cap \ol w_2(O) \neq \varnothing$ implies $\ol w_1 = \ol w_2 \in W_{\mf g \mid \mf h}$.

	It follows that the canonical projection $p \cl \bm B_A \to \bm B_{\ol A}$ is a Galois covering, with automorphisms provided by the monodromy action of $W_{\mf g \mid \mf h}$.
	The choice of the base point $A \in p^{-1}(\mc O_A) \sse \bm B_A$ in the fibre yields an identification $p^{-1}(\mc O_A) \simeq W_{\mf g \mid \mf h}$ between the torsor and the group, so there is a (principle) fibre bundle
	\begin{equation*}
		W_{\mf g \mid \mf h} \lhra \bm B_A \lxra{p} \bm B_{\ol A} \, .
	\end{equation*}
	Then the resulting exact sequence of homotopy groups contains the short sequence
	\begin{equation}
		\label{eq:exact_sequence_full_WMCG}
		1 \lra \pi_1(\bm B_A,A) \lxra{\pi_1(p)} \pi_1(\bm B_{\ol A},\ol A) \lra W_{\mf g \mid \mf h} \lra 1 \, ,
	\end{equation}
	identifying $\pi_0(W_{\mf g \mid \mf h}) \simeq W_{\mf g \mid \mf h}$ (the connecting map is then the monodromy action at the base point).
\end{proof}

\begin{example}
	For example if $A \in \mf t_{\reg}$ then $W_\mf h$ is trivial, so $W_{\mf g \mid \mf h} = W_\mf g$.
	Hence~\eqref{eq:unordered_configuration_space_factor} generalises the generic `configuration' space $\mf t_{\reg} \bs W_\mf g$, and in turn $\pi_1(\bm B_{\ol A},\mc O_A)$ generalises the (full/nonpure) $\mf g$-braid group.
\end{example}

\begin{remark}[Reduced reflection groups]
	At first one may think~\eqref{eq:weyl_group_fission} is the reflection group of the `restricted' hyperplane arrangement
	\begin{equation}
		\label{eq:restricted_arrangement}
		\mc H = \Set{ \Ker(\alpha) \cap U = \Ker \bigl( \eval[1]{\alpha}_U \bigr) | \alpha \in \Phi_\mf g \sm \Phi_\mf h } \sse \mb P \bigl( U^{\dual} \bigr) \, :
	\end{equation}
	however this is \emph{not} the case in general.

	For example if $\Phi_\mf h = A_1 \sse A_2 = \Phi_\mf g$ then~\eqref{eq:weyl_group_fission} is trivial (see \S~\ref{sec:nonpure_low_rank_examples}), while the reduced arrangement is of type $A_1$---so has a Weyl group of order 2.
	The point is there are reflections of~\eqref{eq:restricted_arrangement} which do not come as restrictions of elements in $\Stab_{W_\mf g}(U)$ (see \S~\ref{sec:nonpure_type_A_example}, and the more recent~\cite{doucot_rembado_yamakawa_twisted_local_g_wild_mapping_class_groups}).
\end{remark}

\subsection{General case}
\label{sec:weyl_fission_general}

The results of the previous section can be used inductively to deal with the general case, i.e., with an arbitrary irregular type. Indeed, recall from \cite{doucot_rembado_tamiozzo_2026_local_wild_mapping_class_groups_and_cabled_braids} the usage of \emph{fission} (which is an in the title of~\cite{boalch_2009_through_the_analytic_halo_fission_via_irregular_singularities,boalch_2014_geometry_and_braiding_of_stokes_data_fission_and_wild_character_varieties}): this associates to any irregular type $Q = \sum_{i = 1}^p A_i z^{-i}$ an increasing sequence of (Levi) root subsystems
\begin{equation}
	\label{eq:root_filtration}
	\Phi_{\mf h_1} \sse \dm \sse \Phi_{\mf h_{p+1}} \ceqq \Phi_\mf g \, ,
\end{equation}
as follows.

We consider the sequence of nested centralisers of the coefficients of $Q$, starting from the leading one.
Hence
\begin{equation}
	H_p = \Set{ g \in G | \Ad_g(A_p) = A_p } \sse G \, ,
\end{equation}
which is a connected reductive Lie subgroup.
Then further
\begin{equation}
	H_{p-1} = \Set{ g \in H_p | \Ad_g(A_{p-1}) = A_{p-1} } \sse H_p \, ,
\end{equation}
which coincides with the set of elements of $G$ stabilising both $A_{p-1}$ and $A_p$.
Thus in general $H_i = \Stab_G(A_i,\dc,A_p)$, for $i \in \set{1,\dc,p}$, and at the bottom we find the subgroup $H_1 \sse G$ stabilising the whole of $Q$---which appears in the quasi-Hamiltonian quotients of~\cite{boalch_2014_geometry_and_braiding_of_stokes_data_fission_and_wild_character_varieties}.
By construction $T \sse H_1$, and the inclusion may be proper.
On the opposite end, let us set $H_{p+1} \ceqq G$ for the sake of uniform notation.

\begin{remark}[Fission = breaking]
	This means the structure group of a principal $G$-bundle is `broken', from $G$ down to a reductive subgroup, at any marked point where a connection has an irregular singularity.
\end{remark}

Now set $\mf h_i = \Lie(H_i)$, i.e.,
\begin{equation}
	\mf h_i = \Set{ X \in \mf g | \ad_X(A_i) = \dm = \ad_X(A_p) = 0 } \sse \mf g \, ,
\end{equation}
which is a reductive Lie subalgebra.
In particular $\mf h_1$ is the centraliser of $Q$, and the previous section corresponds to the case where $p = 1$---with $\mf h = \mf h_1$.
Finally, since $\mf t \sse \mf h_i$ (and it is a Cartan subalgebra there), we can consider the corresponding root system $\Phi_{\mf h_i} = \Phi(\mf h_i,\mf t) \sse \Phi_{\mf g}$, leading to~\eqref{eq:root_filtration}.

Now there is a corresponding filtration of Weyl (sub)groups
\begin{equation*}
	W_{\mf h_1} \sse \dm \sse W_{\mf h_{p+1}} = W_\mf g \, ,
\end{equation*}
and~\cite{doucot_rembado_tamiozzo_2026_local_wild_mapping_class_groups_and_cabled_braids} establishes that the deformation space $\bm B_Q$ is a product, with each factor $\bm B_{A_i} \sse \mf t$ determined as in~\eqref{eq:true_complement} by the fission $\Phi_{\mf h_i} \sse \Phi_{\mf h_{i+1}}$ (it is the space of admissible deformations of $A_i \in \mf t$).
Then $W_\mf g$ acts diagonally on $\bm B_Q \sse \mf t^p$.

Now $w (Q) \in \bm B_Q$ means that $w(A_i) \in \bm B_{A_i}$ for $i \in \set{ 1,\dc,p }$, and this condition can be described recursively using Lem.~\ref{lem:setwise_weyl_stabiliser}.
To this end define a sequence of subgroups
\begin{equation*}
	W_1 \sse \dm \sse W_p \sse W_{\mf g} \, ,
\end{equation*}
as follows.
Set as above $U_i \ceqq \Ker(\Phi_{\mf h_i})$, and then
\begin{equation}
	\label{eq:recursive_stabiliser}
	W_p \ceqq \Stab_{W_\mf g}(U_p) \, , \qquad W_{i-1} \ceqq \Stab_{W_i} (U_{i-1}) \sse W_i \, , \qquad  i \in \set{2,\dc,p} \, .
\end{equation}
Denote then by $\mc O_Q$ the orbit of the irregular type under the action of the smallest group $W_1 \sse W_{\mf g}$.

\begin{lemma}
	One has $W_{i-1} = \Stab_{W_i}(\bm B_{A_i})$ for $i \in \set{1,\dc,p}$, and $(W_{\mf g} Q) \cap \bm B_Q = \mc O_Q$.
\end{lemma}

\begin{proof}
	First $w(A_p) \in \bm B_{A_p}$ if and only if $w \in W_p$, and the first statement has been proven in Lem.~\ref{lem:setwise_weyl_stabiliser}---for $i = p$.

	Then we can replace $(\mf h_p,\mf g)$ with $(\mf h_{p-1},\mf h_p)$, and repeat the same construction: we need $w \in W_p$ such that $w(A_{p-1}) \in \bm B_{A_{p-1}}$, where
	\begin{equation*}
		\bm B_{A_{p-1}} = U_{p-1} \cap \bigcap_{\Phi_{\mf h_p} \sm \Phi_{\mf h_{p-1}}} \bigl( \mf t \sm \Ker(\alpha) \bigr) \sse \mf t \, ,
	\end{equation*}
	by~\eqref{eq:true_complement}.
	Reasoning as in the proof of Lem.~\ref{lem:setwise_weyl_stabiliser} this requires $w(\Phi_{\mf h_{p-1}}) \sse \Phi_{\mf h_{p-1}}$, which is equivalent to preserving the partition $\Phi_{\mf h_p} = \Phi_{\mf h_{p-1}} \cup ( \Phi_{\mf h_p} \sm \Phi_{\mf h_{p-1}})$, since by (recurrence) hypothesis $w(\Phi_{\mf h_p}) \sse \Phi_{\mf h_p}$.
	Hence $w \in W_{p-1}$ and
	\begin{equation*}
		W_{p-1} = \Stab_{W_p}(\bm B_{A_{p-1}}) \sse W_p \, .
	\end{equation*}

	Descending until $i = 1$ shows that $w(Q) \in \bm B_Q$ if and only if $w \in \bigcap_i W_i = W_1$, and proves the first statement---inductively.
\end{proof}

Note $W_1 \sse W_\mf g$ is determined by the flag of kernels
\begin{equation}
	\label{eq:kernel_flag}
	\bm U = (\mf t \supseteq U_1 \supseteq \dm \supseteq U_p) \, ,
\end{equation}
by~\eqref{eq:recursive_stabiliser}.
Indeed consider the (parabolic) stabiliser of~\eqref{eq:kernel_flag}, within $W_\mf g$, i.e.
\begin{equation}
	\label{eq:setwise_weyl_stabiliser_general}
	\Stab_{W_\mf g}(\bm U) \ceqq \bigcap_i \Stab_{W_\mf g}(U_i) \sse W_\mf g \, ,
\end{equation}
which coincides with $W_\mf g \cap \Stab_{\GL(\mf t)} (\bm U)$; then:
\begin{lemma}
	\label{lem:recursive_stabiliser_equals_flag_stabiliser}

	One has $W_1 = \Stab_{W_\mf g}(\bm U)$.
\end{lemma}

\begin{proof}
	Postponed to~\ref{proof:lem_recursive_stabiliser_equals_flag_stabiliser}.
\end{proof}

Thus the restriction of orbits to $\bm B_Q$ is controlled by the action of the setwise Weyl-stabiliser of the kernel flag~\eqref{eq:kernel_flag}, generalising the inductive step.

\begin{remark}
	Beware however it is \emph{not} true in general that $\Stab_{W_\mf g}(U_i) = \Stab_{W_\mf g}(\bm B_{A_i})$: e.g. for the fission $\varnothing = \Phi_{\mf h_1} \sse \Phi_{\mf h_2} \sse \Phi_{\mf g}$ one has $U_1 = \mf t$, so $\Stab_{W_\mf g}(U_1) = W_{\mf g}$; but
	\begin{equation*}
		\bm B_{A_1} = \mf t \,\mathbin{\big\backslash} \, \bigcup_{\Phi_{\mf h_2}} \Ker(\alpha) \sse \mf t \, ,
	\end{equation*}
	which is not stabilised by $W_{\mf g}$ if $\mf h_2 \sse \mf g$ is a proper Lie subalgebra.
\end{remark}

Analogously we can identify the subgroup fixing all the admissible deformations of the starting irregular type:

\begin{lemma}
	\label{lem:parabolic_weyl_group_general}

	One has $W_{\mf h_1} = (W_\mf g)_{U_1} = (W_\mf g)_{Q'}$,\fn{
		Note in the rightmost identity we consider two different actions of $W_\mf g$: the former is an action on $\mf t$, the latter on $\mf t \ots \ms{T}_{\Sigma,a}$.} for all $Q' \in \bm B_Q$.
\end{lemma}

\begin{proof}
	Write $Q' = \sum_{i = 1}^p A'_i z^{-i}$, with coefficients $A'_1,\dc,A'_p \in \mf t$.
	By definition $w(Q') = Q'$ if and only if $w(A'_i) = A'_i$ for $i \in \Set{1,\dc,p}$, i.e., $w \in \bigcap_i (W_\mf g)_{A'_i}$.

	Now the argument of Lem.~\ref{lem:parabolic_weyl_group} yields the inductive step for the proof of the identity
	\begin{equation*}
		\bigcap_{j \leq i \leq p} (W_{\mf g})_{A'_i} = (W_{\mf g})_{\set{A'_j,\dotsc,A'_p}} = W_{\mf h_j} \sse W_\mf g \, , \qquad j \in \set{1,\dc,p} \, ,
	\end{equation*}
	whence
	\begin{equation*}
		(W_{\mf g})_{Q'} = W_{\mf h_1} \sse W_\mf g \, .
	\end{equation*}

	On the other hand
	\begin{equation*}
		W_{\mf h_1} = \bigcap_i W_{\mf h_i} \sse \bigcap_i (W_\mf g)_{U_i} = (W_{\mf g})_{U_1} \, ,
	\end{equation*}
	since $U_1 = \sum_i U_i \sse \mf t$, and $W_{\mf h_i}$ acts as the identity on $U_i = \Ker(\Phi_{\mf h_i})$.

	Finally if $w$ acts as the identity on $U_1$ then it also fixes (pointwise) $\bm B_{A_i} \sse U_1$ for $i \in \set{1,\dc,p}$, thus
	\begin{equation*}
		(W_{\mf g})_{U_1} \sse (W_\mf g)_{\bm B_Q} \sse (W_\mf g)_{Q'} \, ,
	\end{equation*}
	proving the remaining inclusion.
\end{proof}

In particular by Lemm.~\ref{lem:recursive_stabiliser_equals_flag_stabiliser}--\ref{lem:parabolic_weyl_group_general} we also have an inclusion $W_{\mf h_i} \sse W_i$, since
\begin{equation*}
	W_{\mf h_i} = (W_\mf g)_{U_i} \sse \Stab_{W_\mf g}(U_j) \, , \qquad j \in \set{i,\dc,p} \, ,
\end{equation*}
as $U_i = \bigcup_{j \geq i} U_j \sse \mf t$.

It follows that $W_{\mf h_1}$ is a normal subgroup of $\Stab_{W_\mf g}(U_1)$, hence a fortiori of ~\eqref{eq:setwise_weyl_stabiliser_general}, and we consider again the quotient group
\begin{equation}
	\label{eq:weyl_group_fission_general}
	W_{\mf g \mid \bm{\mf h}} \ceqq \Stab_{W_\mf g}(\bm U) \bs W_{\mf h_1} \, .
\end{equation}
Note the numerator of~\eqref{eq:weyl_group_fission_general} depends on the whole sequence $\bm{\mf h} = (\mf h_1,\dc,\mf h_p)$, while the denominator only depends on the last term---the pointwise stabiliser of a flag/filtration only depend on its union/sum, contrary to the setwise stabiliser.

\begin{example}[Complete fission and generic case]
	In particular if the fission is `complete', which means that $H_1 = \Stab_G (Q) = T \sse G$ is the maximal torus, then $\Phi_{\mf h_1} = \varnothing$; in this case $U_1 = \mf t$ and $W_{\mf h_1}$ is trivial, so $W_{\mf g \mid \bm{\mf h}} \simeq \Stab_{W_\mf g}(\bm U)$.

	If further $A_p \in \mf t_{\reg}$, then $\bm U$ is stationary at $\mf t$, and $\Stab_{W_\mf g}(\bm U) = W_\mf g$.
\end{example}

Finally by construction there is a homeomorphism $\bm B_{\ol Q} \simeq \bm B_Q \bs W_{\mf g \mid \bm{\mf h}}$, and the same argument of the proof of Prop.~\ref{prop:nonpure_WMCG_factor} yields the following:

\begin{theorem}
	\label{thm:local_wmcg_extension}

	The projection $\bm B_Q \to \bm B_{\ol Q}$ is a Galois covering, and $\Gamma_{\ol Q}$ is an extension of $W_{\mf g \mid \bm{\mf h}}$ by $\Gamma_Q$.
\end{theorem}

Of course if $W_1 = W_{\mf h_1}$ then $\bm B_Q = \bm B_{\ol Q}$, in which case the WMCG is pure (and has been studied in~\cite{doucot_rembado_tamiozzo_2026_local_wild_mapping_class_groups_and_cabled_braids}).

\section{Low-rank examples}
\label{sec:nonpure_low_rank_examples}

Analogously to the pure case, we provide examples of WMCGs for low-rank Lie algebras, after proving we can reduce to the \emph{simple} case.

\subsection{Reduction to the simple case}

Suppose $\mf g = \bops_i \mf I_i$ is a decomposition of $\mf g$ into mutually commuting ideals.
Choose then a root subsystem $\Phi \sse \Phi_{\mf g}$.

Introduce $\mf t_i \ceqq \mf t \cap \mf I_i$ (a Cartan subalgebra of $\mf I_i$), and let $\Phi_{\mf I_i} = \Phi(\mf I_i,\mf t_i) \sse \Phi_{\mf g}$ be the associated root system; there are two other decompositions:
\begin{equation*}
	\mf t = \bops_i \mf t_i \, , \qquad \Phi_{\mf g} = \bops_i \Phi_{\mf I_i} \, .
\end{equation*}
Further let $\Phi^{(i)} \ceqq \Phi \cap \Phi_{\mf I_i}$, which is a root subsystem of $\Phi_{\mf I_i}$.

Then one can show~\cite{doucot_rembado_tamiozzo_2026_local_wild_mapping_class_groups_and_cabled_braids} that the deformation space~\eqref{eq:true_complement} splits as a product $\bm B_Q = \prod_i (\bm B_Q)_i$, where
\begin{equation}
	\label{eq:true_complement_component}
	(\bm B_Q)_i = \Ker \bigl( \Phi^{(i)} \bigr) \cap \bigcap_{\Phi_{\mf I_i} \sm \Phi^{(i)}} \bigl( \mf t_i \sm \Ker(\alpha) \bigr) \sse \mf t_i \, .
\end{equation}

Finally the Weyl group also splits as a product $W_{\mf g} = \prod_i W_{\mf I_i} \sse \prod_i \GL(\mf t_i)$, where the $i$-th factor (the Weyl group of $\Phi_{\mf I_i}$) acts trivially on the complementary direct summands~\cite[Ch.~6, \S~1.2]{bourbaki_1968_groupes_et_algebres_de_lie_chapitres_iv_v_vi}.
It follows that every $W_{\mf g}$-orbit (inside $\mf t$) splits as a product of $W_{\mf I_i}$-orbits (inside $\mf t_i$), so the previous discussion of setwise/pointwise stabilisers can be carried over factorwise, and:

\begin{corollary}
	The deformation space $\bm B_{\ol Q}$ decomposes as a (topological) product $\prod_i (\bm B_{\ol Q})_i$, where $(\bm B_{\ol Q})_i$ is the topological quotient of~\eqref{eq:true_complement_component} with respect to the equivalence relation~\eqref{eq:equivalence_relation}---with $W_{\mf I_i}$ replacing $W_{\mf g}$.
\end{corollary}

In particular the factor corresponding to the centre $\mf Z_{\mf g} \sse \mf g$ is contractible, and can be removed; and further if $\mf g$ is semisimple then the WMCG is a direct product of the groups associated with its simple ideals.

\begin{center}
	\textbf{Hereafter we thus assume that $\mf g$ is simple.}
\end{center}

\subsection{Rank one}

If $\rk(\mf g) = 1$ then the only possible nontrivial fission $\Phi_\mf h \sse \Phi_\mf g$ is the `generic' one $\varnothing \sse \Phi_\mf g$, so $\Gamma_{\ol Q}$ is either trivial or isomorphic to the $\mf g$-braid group.
This is of type $A_1$, i.e., the braid group on 2-strands, and~\eqref{eq:exact_sequence_full_WMCG} becomes
\begin{equation*}
	1 \lra \mb Z \lra \mb Z \lra \mb Z \bs 2\mb Z \lra 1 \, ,
\end{equation*}
with $\Gamma_Q \simeq \mb Z \simeq \Gamma_{\ol Q}$, matching up with a particular case of~\eqref{eq:exact_sequence_braid_group}:
\begin{equation*}
	1 \lra \PBr_2 \lra \Br_2 \lra \on S_2 \lra 1 \, .
\end{equation*}

Equivalently up to homotopy we have $\bm B_Q \simeq \mb S^1$, and the arrow $\bm B_Q \to \bm B_{\ol Q}$ is the 2-sheeted covering of the circle over itself.

\subsection{Rank 2}

Suppose now $\rk(\mf g) = 2$: since $\mf g$ is simple then $\Phi_\mf g$ is isomorphic to $A_2$, $B_2/C_2$ or $G_2$, and $W_\mf g$ is isomorphic to $\Dih_3 \simeq \on S_3$, $\Dih_4$ or $\Dih_6$, respectively (i.e., the symmetries of a triangle, a square, or a hexagon).
Here $\Dih_n$ denotes the dihedral group of order $2n$, for an integer $n \geq 1$---i.e., we use the `geometric' convention rather than the `algebraic' one.

Let $\Delta_{\mf g} \sse \Phi_{\mf g}$ be a base of simple roots.
The generic fission is $\varnothing \sse \Phi_\mf g$, in which case we obtain the $\mf g$-braid group, while the nongeneric (incomplete) fission is $\Phi_\mf h \sse \Phi_\mf g$, with $\Phi_\mf h = \set{\pm \theta}$ for some $\theta \in \Delta_\mf g$; this corresponds to the deformation space $\bm B_Q = \mb C \sm \set{0}$.
With the usual notation we find $U = \Ker(\theta)$ and $W_\mf h \simeq \mb Z \bs 2\mb Z$, and we must describe $\Stab_{W_\mf g}(U) \sse W_\mf g$---acting on $\bm B_Q$.
This is the same as the setwise stabiliser of the line $\mb C \theta \sse \mf t^{\dual}$ for the dual action, and the difference among the three types is due to the parity of the corresponding dihedral group.

Namely for type $A$ the Weyl group yields the standard permutation action of $\on S_3$ on $\mb C^3 \supseteq \mf t^{\dual}$---identified with the standard dual Cartan subalgebra for $\mf{gl}_3(\mb C) \supseteq \mf{sl}_3(\mb C)$.
Then the only nontrivial permutation fixing the line generated by either simple root is the associated (simple) reflection.
It follows that $\Stab_{W_\mf g}(U) = W_\mf h$, so~\eqref{eq:weyl_group_fission} is trivial and the WMCG is pure: it is thus infinite cyclic.

For type $B$ the long roots are vertices of a square centered at the origin of $\mf t_{\mb R}^{\dual} \simeq \mb R^2$, while the short roots are vertices of a smaller square obtained by taking midpoints of each side:
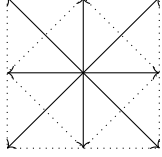
\begin{center}
	\begin{tikzpicture}
		\foreach \th in {0,90,180,-90}
		{
		\draw[->]  (0,0) to (\th:1);
		\draw[dotted] (\th:1) -- ({\th+90}:1);
		\draw[->] (0,0) to ({\th+45}:{sqrt(2)});
		\draw[dotted] ({\th-45}:{sqrt(2)}) -- ({\th+45}:{sqrt(2)});
		}
	\end{tikzpicture}
\end{center}

The Weyl group acts by preserving both squares, and operates as the group of their symmetries.
In both cases a diagonal is fixed by the subgroup generated by the (simple) reflection along the corresponding axis, but also by a rotation of $\pi$.
This means the stabiliser is always the Klein four-group $K_4 \simeq \mb Z \bs 2\mb Z \times \mb Z \bs 2\mb Z$, hence~\eqref{eq:weyl_group_fission} becomes
\begin{equation*}
	W_{\mf g \mid \mf h} \simeq \mb Z \bs 2\mb Z \, ,
\end{equation*}
acting as the antipode on the punctured plane, and $\bm B_Q \to \bm B_{\ol Q}$ is again a two-sheeted covering of the circle onto itself (up to homotopy equivalence).
In particular $\Gamma_{\ol Q}$ is infinite cyclic.

Finally type $G$ yields an analogous situation.
Long/short roots assemble into two Weyl-invariant hexagons in the real plane, and the action of the Klein group (within $\Dih_6$) fixes any given diagonal within each hexagon:
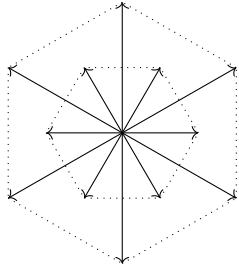
\begin{center}
	\begin{tikzpicture}
		\foreach \th in {0,60,120,180,-120,-60}
		{
		\draw[->]  (0,0) to (\th:1);
		\draw[dotted] (\th:1) -- ({\th+60}:1);
		\draw[->] (0,0) to ({\th+30}:{sqrt(3)});
		\draw[dotted] ({\th+30}:{sqrt(3)}) -- ({\th+90}:{sqrt(3)});
		}
	\end{tikzpicture}
\end{center}

Then we can extend to the complete (nongeneric) fission $\varnothing = \Phi_\mf t \sse \Phi_\mf h \sse \Phi_\mf g$, with the middle term as above.
The associated kernel flag is $\bm U = \bigl( \mf t \supseteq \mf t \supseteq \Ker(\theta) \bigr)$, so the setwise stabiliser stays the same; but this time the irregular type is centralised by the maximal torus only, so $W_{\mf g \mid \bm{\mf h}}$ will either have order 2 (for type $A$), or be isomorphic to the Klein group (for type $B/C$ and $G$).
The result is a covering
\begin{equation*}
	\mb C^{\ast} \times \mb C^{\ast} \simeq \bm B_Q \lra \bm B_{\ol Q} \, , \qquad \mb C^{\ast} = \mb C \sm \set{0} \, ,
\end{equation*}
with either 2 or 4 sheets, and $\Gamma_{\ol Q}$ is an extension of the monodromy group by $\Gamma_Q \simeq \mb Z^2$.

\section{Type A}
\label{sec:nonpure_type_A_example}

In this section we will explicitly describe WMCGs for the special/general linear Lie algebras, in full generality, building on~\cite{doucot_rembado_tamiozzo_2026_local_wild_mapping_class_groups_and_cabled_braids}.

Let $n \geq 2$ be an integer and $\mf g = \mf{sl}_n(\mb C)$.
The Weyl group $W_\mf g \simeq \on S_n$ acts naturally on $V \ceqq \mb C^n$, and we will use the vector representation $\mf g \hra \mf{gl}(V) \simeq \mf{gl}_n(\mb C)$.

\begin{remark}[General linear description]
	\label{rem:gl_n_description}

	Using the basis we identify $V$ with the standard Cartan subalgebra of $\mf{gl}(V)$, so $\mf t = V \cap \mf g$ (the standard Cartan subalgebra of $\mf g$) becomes the subspace of $n$-tuples of points of the complex plane with vanishing barycentre.

	The resulting inclusion $\mb C^{n-1} \simeq \mf t \hra V$ induces a homotopy equivalence
	\begin{equation*}
		\mf t_{\reg} \simeq \Conf_n = \mb C^n \, \mathbin{\big\backslash} \, \bigcup_{1 \leq i \neq j \leq n} H_{ij} \, ,
	\end{equation*}
	using the notation of~\eqref{eq:type_A_complement}, which moreover is compatible with the Weyl group action.
	Hence there is a second homotopy equivalence
	\begin{equation*}
		\mf t_{\reg} \bs W_\mf g \simeq \UConf_n \, ,
	\end{equation*}
	in the notation of~\eqref{eq:unordered_configuration_space}, and whenever useful we will work within the general linear Lie algebra.
\end{remark}

\subsection{Inductive step}
\label{sec:nonpure_type_A_example_step}

If $\Phi_\mf h \sse \Phi_\mf g = A_{n-1}$ is a Levi subsystem, we have an associated $J$-partition $\ul n = \coprod_{j \in J} I_j$, and
\begin{equation*}
	\Phi_\mf h \simeq \bops_{j \in J} A_{\abs{I_j} - 1} \subseteq A_{n-1} \, ,
\end{equation*}
with the usual convention that $A_0 = \varnothing$ (cf.~\cite{doucot_rembado_tamiozzo_2026_local_wild_mapping_class_groups_and_cabled_braids} and \S~\ref{sec:app_notions}).
Namely for $i \in \ul n$ we set
\begin{equation*}
	I_i \ceqq \set{i} \cup \Set{j \in \ul n | \pm \alpha^-_{ij} \in \Phi_{\mf h} } \sse \ul n \, , \qquad \alpha^-_{ij} = e^{\dual}_i - e^{\dual}_j \in \mf t^{\dual} \, ,
\end{equation*}
where $e_i^{\dual} \in V^{\dual}$ is an element of the canonical dual basis, and
\begin{equation*}
	J \ceqq \Set{ \min (I_i) | i \in \ul n } \sse \ul n \, .
\end{equation*}

The Weyl group of $\mf h$ thus comes with a natural factorisation
\begin{equation*}
	W_\mf h \simeq \prod_{j \in J} \on S_{I_j} \sse \on S_n = W_\mf g \, ,
\end{equation*}
with trivial factors corresponding to the trivial components of $\Phi_\mf h$.
The setwise stabiliser of $U = \Ker(\Delta_\mf h) \sse \mf t$ is bigger in general, since we can also permute components of $\Phi_\mf h$ of the same rank.

To state this precisely consider two nonempty finite sets $I$ and $K$, and suppose $I = \coprod_{k \in K} I_k$ is a $K$-partition of $I$ with parts $I_k \sse I$ of \emph{equal} cardinality $m \geq 1$---so $\abs I = m \abs K$.
Then the symmetric group $\on S_I$ contains the subgroup
\begin{equation*}
	N \ceqq \Set{ \tau \in \on S_I | \tau(I_k) \sse I_k \text{ for } k \in K } \simeq (\on S_m)^{\, \abs K} \, ,
\end{equation*}
which stabilises all parts (and permutes their elements).
If $I$ has a total order then there is a `complementary' subgroup $P \sse \on S_I$, which permutes all parts (fixing their elements): more precisely, if $I_k = \bigl( i_1^{(k)},\dc,i_m^{(k)} \bigr) \sse I$ for $k \in K$, then any element of $\sigma \in \on S_K \simeq P$ acts as
\begin{equation}
	\label{eq:part_permutations}
	\sigma \cl i_j^{(k)} \lmt i_j^{(\sigma_k)} \in I \, , \qquad j \in \set{1,\dc,m} \, .
\end{equation}

By construction $N \cap P = 1$ inside $\on S_I$, and $P$ acts on $N$ by conjugation.

\begin{lemma}[Cf.~\cite{yau_2021_infinity_operads_and_monoidal_categories_with_group_equivariance}, Lem.~3.2.8]
	\label{lem:semidirect_product}

	If $\bm{\tau} = \prod_K \tau^{(k)} \in N$, with $\tau^{(k)} \in \on S_{I_k} \simeq \on S_m$, then
	\begin{equation}
		\label{eq:permutation_conjugation_action_symmetric_group}
		\sigma \bm{\tau} \sigma^{-1} = \prod_K \tau^{(\sigma^{-1}_k)} \in N \, , \qquad \sigma \in P \, .
	\end{equation}
\end{lemma}

Hence we have an \emph{inner} semidirect product $P \lts N \sse \on S_I$, and it follows that $(P \lts N) \bs N \simeq P$ canonically.
Equivalently the \emph{outer} semidirect product of $P$ and $N$, with respect to the action~\eqref{eq:permutation_conjugation_action_symmetric_group}, comes with a natural group embedding $P \lts N \hra S_I$.
This latter is also the wreath product $\on S_K \wr \, \on S_m \hra \on S_I$, cf.~\S~\ref{sec:app_notions}.\fn{
	This is a particular example of application of the operadic composition of the symmetric group operad, cf.~\cite[\S~3.1]{yau_2021_infinity_operads_and_monoidal_categories_with_group_equivariance} and see below.}

\begin{remark}
	One has
	\begin{equation*}
		\abs{ P \lts N} = (m!)^{\, \abs K} \abs K ! \leq \bigl( m \abs K \, \bigr) ! = \abs{ \, \on S_I } \, ,
	\end{equation*}
	with strict inequality if $1 < m < \abs I$.
	Thus the embedding $P \lts N \hra \on S_I$ is \emph{proper} for nontrivial partitions of $I$.
\end{remark}

Let us apply this to the present situation:
for an integer $i \geq 1$ write
\begin{equation*}
	K_i \ceqq \Set{ j \in J | \abs{I_j} = i } \subseteq J \, .\fn{
		The notation can be extended to $i = 0$,
		whence however $K_i = K_0 = \vn$ does not contribute.}
\end{equation*}
If $i \geq 2$, the integer $\abs{K_i} \geq 0$ is thus the multiplicity of $A_{i-1}$ as an irreducible component of $\Phi_\mf h$.
For $i = 1$,
instead,
one has a natural bijection
\begin{equation*}
	K_1 \lxra{\simeq} \Set{ i \in \ul n | \pm \alpha^-_{ij} \not\in \Phi_\mf h \text{ for any } j } \, , \qquad I_j = \set{i} \lmt i \, .
\end{equation*}
The subgroups $P_i \simeq \on S_{K_i}$ and $N_i \simeq (\on S_i)^{\, \abs{K_i}}$ of $\on S_J$ are defined as above---%
note that $J$ inherits a natural total order as a subset of $\ul n$,
and that $K_i = \vn$ for all but finitely many integers $i \geq 1$.

\begin{proposition}
	\label{prop:kernel_stabiliser_type_A}

	There is a group isomorphism
	\begin{equation}
		\label{eq:kernel_stabiliser_type_A}
		\Stab_{W_\mf g}(U) \simeq \prod_{i \geq 1} \bigl( \on S_{K_i} \wr \, \on S_i \bigr) \sse W_\mf g \, .
	\end{equation}
\end{proposition}

\begin{proof}
	The statement is the algebraic rewriting of the following claim: the setwise stabiliser of $U = \Ker(\Delta_\mf h)$ is the subgroup of $W_\mf g = \on S_n$ that permutes parts $I_j \sse \ul n$ of the same cardinality, and that further permutes the elements within each part.
	By the above discussion this yields the direct product~\eqref{eq:kernel_stabiliser_type_A}---as permutations of disjoint parts commute.

	To prove the claim, recall the `extended' kernel $\wt{\Ker}(\Delta_\mf h) \sse V$ (in the general linear case) is defined by the condition that the coordinate of any vector are equal within each part $\mb C^{I_j} \sse V$, so its setwise stabiliser is given by the above condition.
	Thus to conclude it is enough to show that the setwise Weyl-stabiliser of the `essential' kernel $U = \wt{\Ker}(\Delta_\mf h) \cap \mf t$ (in the special linear case) is the same; but by construction
	\begin{equation*}
		\wt{\Ker}(\Delta_\mf h) = U \ops \mb{C} \Id_V \sse \mf{gl}(V) \, ,
	\end{equation*}
	and the centre is fixed (pointwise) by $W_{\mf g}$ (cf.~Rmk.~\ref{rem:gl_n_description}).
\end{proof}

\begin{corollary}
	\label{cor:reduced_weyl_group_type_A}

	One has $W_{\mf g \mid \mf h} \simeq \prod_{i \geq 1} \on S_{K_i}$.
\end{corollary}

\begin{proof}
	By definition $W_{\mf g \mid \mf h}$ is the quotient~\eqref{eq:weyl_group_fission}, which is readily computed in this case using Prop.~\ref{prop:kernel_stabiliser_type_A} and the factorisation
	\begin{equation}
		\label{eq:weyl_group_fission_type_A}
		W_\mf h \simeq \prod_{i \geq 1} (\on S_i)^{\abs{\, K_i}} \, . \qedhere
	\end{equation}
\end{proof}

\begin{remark}
	This corresponds to the fact that $W_{\mf g \mid \mf h}$ is naturally identified with the subgroup permuting parts of equal cardinality, fixing elements within each part.
	In particular the exact group sequence
	\begin{equation*}
		1 \lra W_\mf h \lra \Stab_{W_\mf g}(U) \lra W_{\mf g \mid \mf h} \lra 1
	\end{equation*}
	splits.
	(This underlies a more general statement about the normalisers of parabolic subgroups of finite real reflection groups, cf.~\cite{howlett_1980_normalizers_of_parabolic_subgroups_of_reflection_groups}.)

	Further~\eqref{eq:weyl_group_fission_type_A} is naturally a subgroup of the Weyl group of the `reduced' root system
	\begin{equation*}
		\eval[1]{\Phi_\mf g}_U = \Set{ \eval[1]{\alpha}_U | \alpha \in \Phi_{\mf g} } \simeq A_{\abs J-1} \, ,
	\end{equation*}
	viz. a subgroup of `admissible' permutations---inside $\on S_J$.
\end{remark}

On the whole there is a Galois covering $\bm B_Q \to \bm B_{\ol Q}$ with $\prod_{i \geq 1} \bigl( \, \abs{K_i}! \bigr)$ sheets, and to go further let us work within $\mf{gl}(V) \supseteq \mf g$.
Recall from~\cite{doucot_rembado_tamiozzo_2026_local_wild_mapping_class_groups_and_cabled_braids} that there is a canonical vector space isomorphism $\wt{U} \ceqq \wt{\Ker}(\Delta_\mf h) \simeq \mb C^J$, and by Rmk.~\ref{rem:gl_n_description} the $W_\mf g$-equivariant inclusion $U \hra \wt{U}$ yields homotopy equivalences
\begin{equation*}
	\bm B_Q \simeq \Conf_{\, \abs J} \, , \qquad  \bm B_Q \bs W_{\mf g} \simeq \UConf_{\, \abs J} \, ,
\end{equation*}
with fundamental groups $\PBr_{\; \abs J}$ and $\Br_{\, \abs J}$, respectively.
What we have here is an `intermediate' covering, since it is only $W_{\mf g \mid \mf h}$ that acts (freely) on $\bm B_Q$.

To simplify the notation let us then contemplate the following abstract situation.
For an integer $d \geq 1$ consider the ordered configuration space $Y_d \ceqq \Conf_d \sse \mb C^d$.
Let $I$ be a finite set,
and choose an $I$-partition $\varphi \cl \ul d \thra I$ with parts $I_i = \varphi^{-1}(i) \sse \ul d$, for $i \in I$.
Then there is a natural group embedding
\begin{equation*}
	\on S_\varphi \ceqq \prod_I \on S_{I_i} \lhra \on S_d \, ,
\end{equation*}
obtained by juxtaposing permutations, and we let $X_\varphi \ceqq Y_d \bs \on S_\varphi$ (it is the `semiordered' configuration space): this is the space of configurations of $d = \sum_i \abs{I_i}$ points in the complex plane, such that two of them are indistinguishable if they lie within the same part of the $I$-partition.

To identify the fundamental group recall there is an `augmentation' group morphism $p_d \cl \Br_d \to \on S_d$, with kernel $\pi_1(Y_d) = \PBr_d \sse \Br_d$.

\begin{proposition}
	\label{prop:semipure_braid_group}

	There is a group isomorphism
	\begin{equation*}
		\pi_1(X_\varphi) \simeq \Br_\varphi \sse \Br_d \, , \qquad \Br_\varphi \ceqq p_d^{-1} (\on S_\varphi) \, ,
	\end{equation*}
	and $\Br_\varphi$ is an extension of $\on S_\varphi$ by $\PBr_d$.
\end{proposition}

Here $\Br_\varphi$ is thus the `semipure' braid group of the partition, i.e., the group of braids whose underlying permutation lies within $\on S_\varphi \sse \on S_d$.
It is a.k.a.~the `mixed' braid groups (cf.~\cite[Thm.~4]{manfredini_1997_some_subgroups_of_artin_s_braid_group} for a presentation.)

\begin{proof}
	There are Galois coverings $Y_d \to X_d \eqqcolon \UConf_d$ and $Y_d \to X_\varphi$, and it follows that the induced map $X_\varphi \to X_d$ is a covering (with $\bigl[ \, \on S_d \cl \on S_\varphi \bigr]$ sheets).
	Up to identifying groups and torsors after a suitable choice of base points, this yields a commutative diagram of pointed topological spaces, with (principal) fibre bundles in each row:
	\begin{equation*}
		\begin{tikzcd}
			\on S_\varphi \ar[hook]{r} \ar[hook]{d} & Y_d \ar[two heads]{r} \ar[equal]{d} & X_\varphi \ar[two heads]{d} \\
			\on S_d \ar[hook]{r} & Y_d \ar[two heads]{r} & X_d
		\end{tikzcd} \, .
	\end{equation*}

	In turn this leads to a morphism of short exact group sequences:
	\begin{equation*}
		\begin{tikzcd}[row sep=small]
			& \PBr_d \ar{r} \ar[equal]{dd} & \pi_1(X_\varphi) \ar[r] \ar[hook]{dd} & \on S_\varphi \ar[hook]{dd} \ar{rd} & \\
			1 \ar[ru] \ar[rd] & & & & 1 \\
			& \PBr_d \ar{r} & \Br_d \ar{r}[swap]{p_d} & \on S_d \ar{ur} &
		\end{tikzcd} \, ,
	\end{equation*}
	by which we identify as usual $\pi_1(X_\varphi)$ as a subgroup of $\Br_d$.
	(We use here that $\Ker \bigl(\eval[1]{p_d}_{\Br_\varphi} \bigr) = \Ker(p_d) \cap \Br_\varphi = \PBr_d$.)
	This implies in particular that $\pi_1(X_\varphi) \sse \Br_\varphi$,
	and the opposite inclusion concludes the proof:
	but an element of $\Br_\varphi \sse \Br_d$ is precisely---%
	the homotopy class of---%
	a loop in $X_d$ which lifts to a loop in $X_\varphi$,
	which therefore lies in the image of the group embedding $\pi_1(X_\varphi) \hra \Br_d$---%
	by the standard Galois correspondence for covering spaces.\fn{
		All spaces involved are path-connected,
		locally path-connected,
		and semi-locally simply-connected (cf.~\cite[Thm.~1.38]{hatcher_2002_algebraic_topology}).}
\end{proof}

In our situation we thus find a group isomorphism
\begin{equation*}
	\pi_1(\bm B_{\ol Q}) \simeq \Br_\varphi \sse \Br_{\, \abs J} \, ,
\end{equation*}
where $\varphi \cl J \thra I$ is the $I$-partition obtained from $J = \coprod_{i \geq 1} K_i$ by removing the empty parts.
I.e.,
we set
\begin{equation}
	I \ceqq \Set{ \abs{I_j} | j \in J} \sse \mb Z_{\geq 1},
	\qquad \varphi(j) \ceqq \abs{I_j}.
\end{equation}
Beware that there are thus two different partitions in play:
the former is a partition of $\set{1,\dc,n}$ determined by a (Levi) root subsystem $\Phi_{\mf h}$ in type $A_{n-1}$ (according to which roots lie in $\Phi_{\mf h}$),
and the latter is a partition of the set of irreducible components of $\Phi_{\mf h}$ (according to their rank).

\begin{remark}
	The extreme cases are: (i) $\on S_\varphi = 1$, where $X_\varphi = Y_d$; and (ii) $\on S_\varphi = \on S_d$, where $X_\varphi = X_d$ is the (fully) unordered configuration space.
	In our setting this means that either no two irreducible components of $\Phi_{\mf h}$ have the same rank, or conversely that they all have the same rank---respectively.
\end{remark}

\begin{remark}
	There is also a different subgroup of $\Br_d$ associated with the partition and projecting onto $\on S_\varphi$, namely $\prod_i \Br_{I_i} \hra \Br_d$: this is the subgroup obtained by juxtaposing $\abs{I}$ braids, each on $\abs{I_i}$ strands.
	However in general the inclusion $\prod_i \Br_{I_i} \sse \Br_\varphi$ is \emph{proper}.
	E.g. $\Br_1 \times \Br_1 \sse \Br_2$ is trivial, while
	\begin{equation*}
		p_2^{-1}(\on S_1 \times \on S_1) = p_2^{-1}(1) = \PBr_2 \simeq \mb {Z} \, .
	\end{equation*}

	This simple example shows the fundamental group of the semiordered configuration space is \emph{not} just the direct product of the corresponding braid groups.
	Indeed it is possible that two points in different parts braid across each other (along a loop in $X_\varphi$), provided that they are not swapped by the underlying permutation of the overall braid.
\end{remark}

\begin{remark}
	By the standard Galois correspondence, the isomorphism class of the covering $X_\varphi \to X_d$ matches up with the conjugacy class of $\Br_\varphi \sse \Br_d$, and the latter is generically nontrivial---%
	as $\on S_\varphi \sse \on S_d$ is typically \emph{not} a normal subgroup.
\end{remark}

\subsection{General case: ranked fission trees}

Suppose now to have an increasing filtration
\begin{equation*}
	\Phi_{\mf h_1} \sse \dm \sse \Phi_{\mf h_p} \sse \Phi_{\mf h_{p+1}} \ceqq A_{n-1} \, .
\end{equation*}
of Levi subsystems.
As in~\cite{doucot_rembado_tamiozzo_2026_local_wild_mapping_class_groups_and_cabled_braids} this corresponds to a (finite,
rooted) `fission' tree $T = (T_0,\bm \phi)$ of height $p \geq 1$ (cf.~\S~\ref{sec:app_notions}).
The set $J_l = J_{\mf h_l}$ is as above, for $l \in \ul p$, and then we add a tree root at level $p+1$.
By definition $\bm \phi(i) = j \in J_{l+1}$ means that the irreducible component of $\mf h_l \sse \mf h_{l+1}$ corresponding to $i \in J_l$ lies within the irreducible component of $\mf h_{l+1}$ corresponding to $j$.

This was enough to encode $\Gamma_Q$, while in the full/nonpure case we must retain more data, according to the results of the previous section.

\begin{definition}[Ranked fission tree]
	\label{def:ranked_fission_tree}

	A \emph{ranked} fission tree is a fission tree $T = (T_0,\bm \phi)$ equipped with a \emph{rank function} $\bm r \cl T_0 \to \mb Z_{\geq 1}$, i.e., a function satisfying
	\begin{equation*}
		\bm r(i) = \sum_{\bm\phi^{-1}(i)} \bm r(j) \, , \qquad i \in T_0 \, .
	\end{equation*}
	Then $r(T) \ceqq \bm r(\ast) \geq 1$ is the \emph{rank} of the tree.
\end{definition}

This means to each node we attach a positive rank, which equals the sum of its child-nodes'.
In particular $\sum_{J_l} \bm r(i) = r(T)$, independently of the level, and the rank function is determined by assigning ranks to the leaves---i.e., by $\eval[1]{\bm r}_{J_1} \in \mb Z_{\geq 1}^{J_1}$.

The algorithm to associate a ranked fission tree $(T,\bm r)$ to~\eqref{eq:root_filtration} is the following: the underlying fission tree is constructed as in~\cite{doucot_rembado_tamiozzo_2026_local_wild_mapping_class_groups_and_cabled_braids}, and we further set $\bm r(i) = k+1$ if the node $i \in J_l$ corresponds to a type-$A$ irreducible rank-$k$ component of $\Phi_{\mf h_{l+1}}$.
Working within the general linear Lie algebra, this is the same as setting $\bm r(i) = k$ if $i$ corresponds to an irreducible component isomorphic to $\Phi_{\mf{gl}_k(\mb C)}$---including $\Phi_{\mf{gl}_1(\mb C)} = \varnothing$. It follows that $r(T) = n$ if we work within $\mf{gl}_n(\mb C)$.

By construction the Weyl group of $\mf h_l \sse \mf g$ comes with a canonical group isomorphism
\begin{equation}
	\label{eq:weyl_groups_tree}
	W_{\mf h_l} \simeq \prod_{i \in J_l} \on S_{\bm r(i)} \, , \qquad l \in \ul p \, ,
\end{equation}
and to construct the stabiliser of the kernel flag in terms of the tree we introduce the following:

\begin{definition}
	An \emph{isomorphism} $(T_0,\bm \phi,\bm r) \lxra{\simeq} (T_0',\bm \phi',\bm r')$ of ranked fission trees is a bijection $\bm f \cl T_0 \to T_0'$ matching roots, and such that there are commutative diagrams:
	\begin{equation*}
		\begin{tikzcd}
			T_0 \sm \set{\ast} \ar{r}{\bm f} \ar{d}[swap]{\bm \phi} & T_0' \sm \set{\ast} \ar{d}{\bm \phi'} \\
			T_0 \ar{r}[swap]{\bm f} & T_0'
		\end{tikzcd}
		\qquad \text{and} \qquad
		\begin{tikzcd}[column sep=small]
			T_0 \ar{rr}{\bm f} \ar{rd}[swap]{\bm r} & & T_0' \ar{dl}[outer sep=-1.5pt]{\bm r'} \\
			& \mb Z_{\geq 1} &
		\end{tikzcd} \, .
	\end{equation*}
	An \emph{automorphism} of $(T_0,\bm \phi,\bm r)$ is an isomorphism $(T_0,\bm \phi,\bm r) \lxra{\simeq} (T_0,\bm \phi,\bm r)$; their group is denoted by $\Aut(T,\bm r)$.
\end{definition}

This restricts the usual notion of isomorphism of (rooted) trees, by further asking that ranks be preserved.
Note by definition an automorphism preserves the nodes at each level, and is uniquely determined by the image of the leaves.

\subsection{General case: reflection groups}

It is possible to compute the automorphism group of the tree recursively, and in turn this will control the monodromy action of the Galois covering $\bm B_Q \to \bm B_{\ol Q} = \bm B_Q \bs W_{\mf g \mid \bm{\mf h}}$.

Choose then a ranked fission tree $(T,\bm r)$, and note its subtrees are equipped with restricted rank functions.
In particular let $\mc T = \mc T(T)$ be the set of (ranked) \emph{maximal} proper subtrees, i.e., the subtrees of $T$ rooted at each child-node of the root;
it can be identified with $J_p$,
and we (noncanonically) fix once and for all a complete set of representatives $\wt{\mc T} \sse \mc T$ of isomorphism classes.
Finally,
denote by $n(t) \geq 1$ the cardinality of the isomorphism class of any maximal proper subtree $t \in \mc T$.

\begin{definition}
	The \emph{extended} automorphism group $\wt{\Aut}(T,\bm r)$ of the ranked fission tree $(T,\bm r)$ is defined recursively by
	\begin{equation}
		\label{eq:recursive_stabiliser_type_A}
		\wt{\Aut}(T,\bm r) = \prod_{t \in \wt{\mc T}} \on S_{n(t)} \wr \, \wt{\Aut} (t,\bm r_0) \, , \qquad \bm r_0 \ceqq \eval[1]{\bm r}_{t_0} \, ,
	\end{equation}
	with basis $\wt{\Aut} \bigl( i,r(i) \bigr) \ceqq \on S_{r(i)}$ for $i \in J_1$.
\end{definition}

A priori~\eqref{eq:recursive_stabiliser_type_A} depends on the choice of $\wt{\mc T} \sse \mc T$, but the following identification in particular shows it does not.

\begin{theorem}
	\label{thm:tree_automorphisms}

	One has $\wt{\Aut} (T,\bm r) = \Stab_{W_\mf g}(\bm U)$.
	Furthermore the automorphism group $\Aut(T,\bm r)$ is obtained recursively as in~\eqref{eq:recursive_stabiliser_type_A}, but with recursion basis $\Aut \bigl(i,r(i) \bigr) \ceqq \on S_1$ for $i \in J_1$---the trivial group.
\end{theorem}

\begin{proof}
	The first item can be proven by induction on $p \geq 1$.
	If $p = 1$ then a maximal proper subtree is a leaf, so $\mc T = J_1$: then two leaves are isomorphic (as ranked trees) if and only if they have the same rank.
	Hence for $i \in J_1$ the integer $n(i) \geq 1$ is the number of rank-$r(i)$ leaves, and in this case
	\begin{equation*}
		\wt{\Aut}(T,\bm r) = \prod_{i \in \wt{J}_1} \on S_{n(i)} \wr \, \on S_{r(i)} \sse \on S_{r(T)} \, ,
	\end{equation*}
	where $\wt{J}_1 \sse J_1$ is a set of representatives of leaves---of all possible ranks.
	The result follows from Prop.~\ref{prop:kernel_stabiliser_type_A}.

	Now let $p \geq 2$.
	By the induction hypothesis, the bases of the wreath products in~\eqref{eq:recursive_stabiliser_type_A} are the setwise stabilisers of the deformation space of the `sub-irregular types' corresponding to the eigenspaces of the leading coefficient.
	In addition to that, we are then permuting isomorphic maximal proper subtrees, i.e., eigenspaces of the leading coefficient whose nested decomposition (into eigenspaces for the subleading coefficients) plays a symmetric role: this yields the whole of $\Stab_{W_\mf g}(\bm U)$, as any other permutation of the eigenvalues of the leading	coefficient moves the irregular type out of the space of admissible deformations.

	The second item is a straightforward extension from the unranked case, and can also be proven recursively on $p \geq 1$.
	If $p = 1$ an automorphism is the data a permutation of the leaves which matches up ranks.
	Hence in this case
	\begin{equation}
		\label{eq:automorphism_are_leaves_permutations}
		\Aut(T,\bm r) = \prod_{i \in \wt{J}_1} \on S_{n(i)} = \prod_{i \in \wt{J}_1} \on S_{n(i)} \wr \, \on S_1 \, ,
	\end{equation}
	using the above notation. (Note $\on S_{n(i)} \sse \on S_{J_1}$ is naturally identified with the symmetric group of rank-$r(i)$ leaves.)

	Now let $p \geq 2$.
	By the induction hypothesis the bases of the wreath products in~\eqref{eq:recursive_stabiliser_type_A} are the automorphism groups of the maximal proper ranked subtrees.
	In addition to that, we are then permuting isomorphic maximal proper subtrees: this yields the whole of $\Aut(T,\bm r)$, as any other permutation of child-nodes of the root, bringing along the corresponding subtrees, cannot restrict to an isomorphism of these latter.
\end{proof}

\begin{example}
	The most symmetric example is that in which $\bm r$ is constant at each level.
	In this case any automorphism of the underlying tree $T = (T_0,\bm \phi)$ preserves the rank function.

	If moreover $T$ is a complete $m$-ary tree, viz. if all interior nodes have $m \geq 1$ child-nodes, then simply
	\begin{equation*}
		\Aut(T,\bm r) \simeq \ub{\on S_m \wr \, \dm \, \wr \, \on S_m}_{p \text{ times }} \, ,
	\end{equation*}
	the $p$-fold wreath power---recall this example of wreath product is associative.
	The extended group instead is
	\begin{equation*}
		\wt{\Aut}(T,\bm r) \simeq (\on S_m)^{\wr p} \wr \, \on S_r \, ,
	\end{equation*}
	where $r \geq 1$ is the rank of any leaf.

	On the opposite end $\bm r$ is injective at each level, so the group $\Aut(T,\bm r)$ is trivial, and $\wt{\Aut}(T,\bm r) \simeq \prod_{J_1} \on S_{r(i)} \sse \on S_{r(T)}$.
	In this case the WMCG is pure.
\end{example}

Now by (recursive) construction $\Aut(T,\bm r) \sse \wt{\Aut}(T,\bm r)$ is a subgroup, so by Thm.~\ref{thm:tree_automorphisms} it can be identified with a subgroup of the kernel-flag stabiliser.

Indeed choose $\bm f \in \Aut(T,\bm r)$, so by definition $\bm f \cl T_0 \to T_0$ yields rank-preserving permutations $f_l \ceqq \eval[1]{\bm f}_{J_l} \in \on S_{J_l}$ of the nodes at each level $l \in \ul{p+1}$.
In particular $f_1$ permutes subsets of leaves (of constant rank), and we can map it to an element of $W_{\mf g} \simeq \on S_{r(T)}$ along the group embedding
\begin{equation*}
	\prod_{\wt{J}_1} \on S_{n(i)} \lhra \prod_{\wt{J}_1} \on S_{n(i)} \wr \, \on S_{r(i)} \sse \on S_{r(T)} \, ,
\end{equation*}
keeping the notation of the proof of Thm.~\ref{thm:tree_automorphisms}.
This yields an injective group morphism $\iota \cl \Aut(T,\bm r) \hra \Stab_{W_\mf g}(\bm U)$---since $\bm f$ is determined by $f_1$.

By construction the image of $\iota$ is disjoint from $W_{\mf h_1} = \prod_{J_1} \on S_{r(i)}$ (using~\eqref{eq:weyl_groups_tree}), and acts on it by conjugation, so there is a second group embedding
\begin{equation}
	\label{eq:semidirect_product_stabiliser_type_A}
	\wt{\iota} \cl \Aut(T,\bm p) \lts W_{\mf h_1} \lhra \wt{\Aut}(T,\bm r) \, .
\end{equation}

\begin{proposition}
	\label{prop:split_sequence_weyl_fission_type_A}

	The group morphism~\eqref{eq:semidirect_product_stabiliser_type_A} is surjective.
\end{proposition}

Hence the exact group sequence
\begin{equation*}
	1 \lra W_{\mf h_1} \lra \Stab_{W_\mf g}(\bm U) \lra W_{\mf g \mid \bm{\mf h}} \lra 1
\end{equation*}
splits, generalising the recursive step.

\begin{proof}
	It is equivalent to show that $W_{\mf h_1} \sse \wt{\Aut}(T,\bm r)$ is a normal subgroup, and that
	\begin{equation*}
		\wt{\Aut}(T,\bm r) \bs W_{\mf h_1} \simeq \Aut(T,\bm r) \, .
	\end{equation*}
	This can be proven recursively on $p \geq 1$, the base being the content of Cor.~\ref{cor:reduced_weyl_group_type_A}.

	If $p \geq 2$ consider a maximal proper subtree $t \in \mc T$: its leaves yield a subset $J_1(t) \sse J_1$, and there is a partition
	\begin{equation*}
		J_1 = \coprod_{\mc T} J_1(t) \, ,
	\end{equation*}
	of the leaves of $T$.
	Accordingly the centraliser of the irregular type splits as
	\begin{equation*}
		W_{\mf h_1} \simeq \prod_{\mc T} W_{\mf h_1}(t) = \prod_{\wt{\mc T}} \bigl(W_{\mf h_1}(t)\bigr)^{n(t)} \, ,
	\end{equation*}
	where $W_{\mf h_1}(t) \eqqcolon \prod_{J_1(t)} \on S_{r(i)}$ is the Weyl group of the Lie algebra $\mf h_1 \cap \mf{gl}_{r(t)}(\mb C) \sse \mf h_1$, and in turn $\mf{gl}_{r(t)}(\mb C) \sse \mf{gl}_{r(T)}(\mb C)$ matches up with the eigenspace of the leading coefficient corresponding to the root of the subtree $t$.

	Hence, using the decomposition~\eqref{eq:recursive_stabiliser_type_A} (and that direct products and quotients commute), the result follows from Lem.~\ref{lem:wreath_quotients} below; indeed in particular
	\begin{align*}
		\wt{\Aut}(T,\bm r) \bs W_{\mf h_1} & \simeq \Biggl( \prod_{\wt{\mc T}} \on S_{n(t)} \wr \, \wt{\Aut}(t,\bm r) \Biggr) \bs W_{\mf h_1} = \prod_{\wt{\mc T}} \Bigl( \on S_{n(t)} \wr \,  \wt{\Aut}(t,\bm r) \bs 1 \wr W_{\mf h_1(t)} \Bigr) \\
		                                   & = \prod_{\wt{\mc T}} \on S_{n(t)} \wr \, \Aut(t,\bm r) = \Aut(T,\bm r) \, ,
	\end{align*}
	by the recursive hypothesis---and definition of the automorphism group.
\end{proof}

\begin{lemma}
	\label{lem:wreath_quotients}

	Let $m \geq 0$ be an integer and $P$ a group, and choose a normal subgroup $N \sse P$.
	Then $1 \wr \, N \sse \on S_m \wr \, P$ is a normal subgroup, and in this identification there is a canonical group isomorphism
	\begin{equation}
		\label{eq:wreath_quotients}
		\bigl( \on S_m \wr \, P \bigr) \bs N \simeq \on S_m \wr \, \bigl( P \bs N \bigr) \, .
	\end{equation}
\end{lemma}

\begin{proof}
	Postponed to \S~\ref{proof:lem_wreath_quotients}.
\end{proof}

Hence in brief there is an explicit (finite) algorithm to compute the `effective' subquotient of the Weyl group acting freely on the deformation space of any type-$A$ irregular type, to yield the deformation space of the associated irregular class.

\begin{example}
	\label{ex:examples_presentations}
	Let us look at the examples of type-$A$ irregular type considered in~\cite{doucot_rembado_tamiozzo_2026_local_wild_mapping_class_groups_and_cabled_braids}, which all had pure WMCG isomorphic to  $\PBr_2 \times \PBr_3^2 \times \PBr_4$.
	We will see their associated stabilisers are \emph{not} all isomorphic.

	Let us work with the Lie group $G = \SL_9(\mb C)$: an irregular type $Q$ is given by a polynomial in the variable $x = z^{-1}$, whose coefficients are traceless diagonal matrices of size $9$.

	First consider
	\begin{equation*}
		Q = A_1 x + A_2 x^2 + A_3 x^3 \, , \qquad A_i \in \mf{sl}_9(\mb C) \, ,
	\end{equation*}
	with
	\begin{align*}
		A_1 & = \diag(4,3,2,1,0,-1,-2,-3,-4) \, , \\
		A_2 & = \diag(4,4,3,2,1,0,-3,-4,-7) \, ,  \\
		A_3 & = \diag(2,2,1,1,1,0,0,0,-7) \, .
	\end{align*}

	The corresponding ranked fission tree $(T,\bm r)$ is in Fig.~\ref{fig:exmp_tree_1}, indicating the rank of each node within the corresponding vertex.
	From the recursive algorithm we get
	\begin{equation*}
		\Aut(T,\bm r) = \on S_2 \times (\on S_2 \wr \on S_3) \, .
	\end{equation*}

	\begin{figure}
		\begin{center}
			\begin{tikzpicture}
				\foreach \name/\x/\y/\r in {A1/1/0/1,A2/2/0/1,A3/3/0/1,A4/4/0/1,A5/5/0/1,A6/6/0/1,A7/7/0/1,A8/8/0/1,A9/9/0/1,
						C1/1.5/2/2,C2/4/2/3,C3/7/2/3,C4/9/2/1,
						B1/1.5/1/2,B2/3/1/1,B3/4/1/1,B4/5/1/1,B5/6/1/1,B6/7/1/1,B7/8/1/1,B8/9/1/1, D1/5/3/9}
				\node[vertex] (\name) at (\x,\y){\r};
				\foreach \from/\to in {A1/B1,A2/B1,A3/B2,A4/B3,A5/B4,A6/B5,A7/B6,A8/B7,A9/B8,
						B1/C1,B2/C2,B3/C2,B4/C2,B5/C3,B6/C3,B7/C3,B8/C4,D1/C1,D1/C2,D1/C3,D1/C4}
				\draw (\from) -- (\to);
			\end{tikzpicture}
		\end{center}
		\caption{
			\label{fig:exmp_tree_1}
			Example of ranked fission tree (I)}
	\end{figure}

	This is the same as the automorphism group of the ranked tree for the irregular type $Q = A_1 x + A_2 x^2$, with
	\begin{equation*}
		A_1 = \diag(4,3,2,1,0,-1,-2,-3,-4) \, , \quad A_2 = \diag(4,1,1,0,0,0,-2,-2,-2) \in \mf{sl}_9(\mb C) \, .
	\end{equation*}
	Indeed in that case the ranked fission tree is as in Fig.~\ref{fig:exmp_tree_2}.

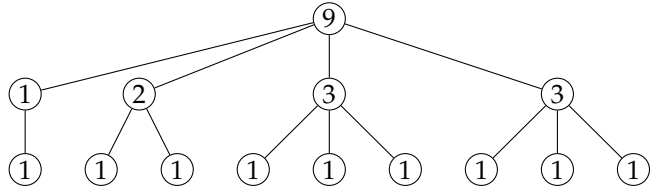
	\begin{figure}
		\begin{center}
			\begin{tikzpicture}
				\foreach \name/\x/\y/\r in {A1/1/0/1,A2/2/0/1,A3/3/0/1,A4/4/0/1,A5/5/0/1,A6/6/0/1,A7/7/0/1,A8/8/0/1,A9/9/0/1,
						B1/1/1/1,B2/2.5/1/2,B3/5/1/3,B4/8/1/3, C1/5/2/9}
				\node[vertex] (\name) at (\x,\y){\r};
				\foreach \from/\to in {A1/B1,A2/B2,A3/B2,A4/B3,A5/B3,A6/B3,A7/B4,A8/B4,A9/B4, C1/B1, C1/B2,C1/B3,C1/B4}
				\draw (\from) -- (\to);
			\end{tikzpicture}
		\end{center}
		\caption{
			\label{fig:exmp_tree_2}
			Example of ranked fission tree (II)}
	\end{figure}

	Finally let us consider $Q = A_1 x + A_2 x^2 + A_3 x^3 + A_4 x^4$, with
	\begin{align*}
		A_1 & = \diag(4,3,2,1,0,-1,-2,-3,-4) \, , \\
		A_2 & =\diag(4,4,3,2,1,0,-3,-4,-7) \, ,   \\
		A_3 & = \diag(2,2,2,2,1,0,-3,-3,-3) \, ,  \\
		A_4 & = \diag(1,1,1,1,1,1,0,-2,-4) \, .
	\end{align*}

	The ranked fission tree is then as in Fig.~\ref{fig:exmp_tree_3}.
	Its automorphism group is now $\Aut(T,\bm r) = \on S_3 \times \on S_2^3$, which is \emph{not} the same as in the previous two cases.
	This showcases how $\Gamma_Q$ only depends on the set of unordered configuration spaces attached to the (unranked) fission tree, while $\Gamma_{\ol Q}$ also depends on their positions in the tree.

	\begin{figure}
		\begin{center}
			\begin{tikzpicture}
				\foreach \name/\x/\y/\r in {A1/1/0/1,A2/2/0/1,A3/3/0/1,A4/4/0/1,A5/5/0/1,A6/6/0/1,A7/7/0/1,A8/8/0/1,A9/9/0/1,
						B1/1.5/1/2,B2/3/1/1,B3/4/1/1,B4/5/1/1,B5/6/1/1,B6/7/1/1,B7/8/1/1,B8/9/1/1,
						C1/2.5/2/4,C2/5/2/1,C3/6/2/1,C4/7/2/1,C5/8/2/1,C6/9/2/1,
						D1/3.5/3/6,D2/7/3/1,D3/8/3/1,D4/9/3/1,
						E1/5/4/9}
				\node[vertex] (\name) at (\x,\y){\r};
				\foreach \from/\to in {A1/B1,A2/B1,A3/B2,A4/B3,A5/B4,A6/B5,A7/B6,A8/B7,A9/B8,
						B1/C1,B2/C1,B3/C1,B4/C2,B5/C3,B6/C4,B7/C5,B8/C6,
						C1/D1,C2/D1,C3/D1,C4/D2,C5/D3,C6/D4,
						D1/E1,D2/E1,D3/E1,D4/E1}
				\draw (\from) -- (\to);
			\end{tikzpicture}
			\caption{
				\label{fig:exmp_tree_3}
				Example of ranked fission tree (III)}
		\end{center}
	\end{figure}

\end{example}

\subsection{General case: (cabled) braid group}

Denote now by $\bm B_Q = \prod_{l = 1}^p \bm B_l$ the deformation space, in the decomposition associated with the fission tree (as in~\cite{doucot_rembado_tamiozzo_2026_local_wild_mapping_class_groups_and_cabled_braids}).
This means $\bm B_l \sse \mb C^{J_l}$ is a product of ordered configuration spaces (= type-$A$ root-hyperplane complements), attached to the nodes at the above level $J_{l+1}$: namely
\begin{equation*}
	\bm B_l = \prod_{J_{l+1}} \Conf_{k_i} \sse \mb C^{J_l} \, , \qquad k_i = \abs{ \bm \phi^{-1}(i) } \geq 1 \, ,
\end{equation*}
counting the number of child-nodes.
Thus globally
\begin{equation}
	\label{eq:deformation_space_tree}
	\bm B_Q = \prod_{T_0} \Conf_{k_i} \sse \mb C^{T_0 \sm \set{\ast}} \, , \qquad \text{and} \qquad \pi_1(\bm B_Q) \simeq \prod_{T_0} \PBr_{k_i} \, .
\end{equation}

Now recall from~\cite[\S~6]{doucot_rembado_tamiozzo_2026_local_wild_mapping_class_groups_and_cabled_braids} that the \emph{cabling} the of pure braid group operad, viz. the operadic composition
\begin{equation}
	\label{eq:pure_braid_group_operad_composition}
	\gamma^{\ms{P\!B}} \cl \PBr_n \times \prod_{i = 1}^n \PBr_{k_i} \lra \PBr_k \, , \qquad k = \sum_i k_i \, ,
\end{equation}
for $n,k_1,\dc,k_n \geq 0$, yields a group embedding $\pi_1(\bm B_Q) \hra \PBr_{J_1}$.
More precisely,
recursive cabling along the (unranked) fission tree $T = (T_0,\bm \phi)$ leads to the pure `cabled braid group' $\ms{P\!B}(T) \sse \PBr_{\, \abs{J_1}}$ of the tree, and there is a group isomorphism $\pi_1(\bm B_Q) \simeq \ms{P\!B}(T)$.\fn{
	See op.~cit. for an explanation of terminology, due to the nested braiding of eigenspaces for the coefficients of the irregular type.}
The point is that in the pure case one finds a `noncrossed' group/action operad, which in particular implies~\eqref{eq:pure_braid_group_operad_composition} is a group morphism equipping the domain with the direct product structure.

Here instead we naturally encountered the operadic composition of the \emph{symmetric group operad} $\ms S = \bigl( \on S_{\bullet}, 1 \in \on S_1, \gamma^{\ms S} \bigr)$, viz. the function
\begin{equation}
	\label{eq:symmetric_group_operad_composition}
	\gamma^{\ms S} \cl \on S_n \times \prod_{i = 1}^n \on S_{k_i} \lra \on S_k \, , \qquad (\sigma,\bm \tau) \lmt \gamma^{\ms S}(\sigma;\bm \tau) \, ,
\end{equation}
for $\sigma \in \on S_n$ and $\bm \tau = \prod_{i = 1}^n \tau^{(i)}$.
Its definition is a generalisation of the above construction to arbitrary partitions $\ul k = \coprod_{i = 1}^n I_i$, where $k_i = \abs{I_i}$.
Namely there is a `block permutation' operation
\begin{equation*}
	\on S_n \lra \on S_k \, , \qquad \sigma \lmt \sigma \braket{ k_1,\dc,k_n} \, ,
\end{equation*}
which consists in the permutation of all parts by fixing their elements, and then
\begin{equation*}
	\gamma^{\ms S}(\sigma;\bm \tau) \ceqq \sigma \braket{ k_1,\dc,k_n} \cdot \bm \tau \in \on S_k \, ,
\end{equation*}
with tacit use of the natural group embedding $\prod_i \on S_{I_i} \hra \on S_k$ on the right factor.

\begin{lemma}
	\label{lem:semidirect_product_symmetric}

	If $k \ceqq k_1 = \dm = k_n$ then~\eqref{eq:symmetric_group_operad_composition} is an injective group morphism,
	equipping the domain with the semidirect product structure where $\on S_n$ acts by permuting the direct factors of $\prod_{i = 1}^n \on S_{k_i} = (\on S_k)^n$.
\end{lemma}

\begin{proof}
	Postponed to~\ref{proof:lem_semidirect_product_symmetric}.
\end{proof}

\begin{remark}
	This means the operadic compositions yields in particular group embeddings $\on S_n \wr \, \on S_k \hra \on S_{k \cdot n}$, which were used above.
\end{remark}

The recursive definition~\eqref{eq:recursive_stabiliser_type_A} now becomes
\begin{equation}
	\label{eq:cabled_symmetric_group}
	\Aut(T,\bm r) = \prod_{t \in \wt{\mc T}} \gamma^{\ms S} \Bigl( \on S_{n(t)} \times \Aut( t,\bm r_0)^{n(t)} \Bigr) \sse \on S_{J_1} \, , \qquad \bm r_0 = \eval[1]{\bm r}_{t_0} \, ,
\end{equation}
starting again from the trivial group at each leaf.
Finally, this can be rewritten to exhibit the relation with braid groups.

\begin{lemma}
	\label{lem:direct_semidirect_product}

	Let $\set{P_i}_{i \in I}$ and $\set{N_i}_{i \in I}$ be two collections of groups indexed by a finite set $I$,
	and let $\rho_i \cl P_i \to \Aut(N_i)$ be group morphisms---%
	for $i \in I$.
	Then there is a canonical group isomorphism
	\begin{equation*}
		P \lts N \simeq \prod_I P_i \lts N_i \, , \quad P = \prod_I P_i \, , \quad N = \prod_I N_i \, ,
	\end{equation*}
	using the product action on the left-hand side:
	\begin{equation*}
		\rho \cl P \lra \prod_I \Aut(N_i) \sse \Aut(N) \, , \qquad \rho = \prod_I \rho_i \, .
	\end{equation*}
\end{lemma}

\begin{proof}
	Postponed to \S~\ref{proof:lem_direct_semidirect_product}.
\end{proof}

It follows from Lem.~\ref{lem:direct_semidirect_product} that~\eqref{eq:cabled_symmetric_group} is equivalent to the recursive definition
\begin{equation}
	\label{eq:automorphism_group_true_recursion}
	\Aut(T,\bm r) = \on S_\varphi \lts \prod_{\wt{\mc T}} \Aut(t,\bm r_0)^{n(t)} \, ,
\end{equation}
introducing the $\wt{\mc T}$-partition $\varphi = \varphi_p \cl J_p \thra \wt{\mc T}$ induced from the isomorphism classes of maximal proper subtrees; this means that $\on S_\varphi = \prod_{\wt{\mc T}} \on S_{n(t)} \sse \on S_{J_p}$.

\begin{remark}
	\label{rmk:single_iso_class}

	Suppose for simplicity that $\wt{\mc T}$ is a singleton,
	i.e.,
	that all maximal subtrees are isomorphic.
	If $n \ceqq \abs{J_p}$ we then have $\on S_\varphi \simeq \on S_n$ and
	\begin{equation}
		\Aut(T,r) = \on S_n \wr \Aut(t,\bm r_0),
	\end{equation}
	as a particular case of~\eqref{eq:automorphism_group_true_recursion}.

	Moreover,
	taking the leaves of each maximal subtree of $T$ yields a partition of the set $J_1$ of leaves of $T$:
	all parts have the same cardinality,
	say $k \geq 1$,
	so that $\abs{J_1} = k \cdot n$.
	In this case the permutation action of $\on S_n$ on $J_p \simeq \mc T$ matches up with block permutations of the leaves,
	and this corresponds to the usual group embedding $\on S_n \sse \on S_{J_1} \simeq \on S_{n \cdot k}$.
	In turn,
	as above one has $\Aut(t,\bm r_0) \sse \on S_k$,
	so that overall $\Aut(t,\bm r_0)^n \sse \on (S_k)^n \sse \on S_{n \cdot k}$ and
	\begin{equation}
		\Aut(T,\bm r) \sse \on S_n \wr \on S_k \sse \on S_{n \cdot k}.
	\end{equation}
	In brief,
	this makes it possible to view the automorphism group of $(T,r)$ as an \emph{inner} semidirect product inside a large symmetric group.

	The general case follows by taking direct products over $\wt{\mc T}$,
	as the factors determine disjoint subsets of leaves,
	whence commuting subgroups of $\on S_{J_1}$.
\end{remark}

The expression~\eqref{eq:automorphism_group_true_recursion} clarifies the natural definition of an analogous (full/nonpure) `cabled' braid groups: one should `lift' this through the (augmentation) operad morphism $\bm p \cl \ms B \to \ms S$, where $\ms B = \bigl( \Br_{\bullet},1 \in \Br_1,\gamma^{\ms B} \bigr)$ is the (full/nonpure) braid group operad.

\begin{definition}
	\label{def:cabled_braid_group}

	The \emph{cabled braid group} of the ranked tree $(T,r)$ is the group recursively defined by
	\begin{equation}
		\label{eq:cabled_braid_group}
		\ms B(T,r) = \Br_\varphi \lts \prod_{\wt{\mc T}} \ms B (t,\bm r_0)^{n(t)} \sse \Br_{\, \abs{J_1}} \, ,
	\end{equation}
	with basis $\ms B \bigl( i,r(i) \bigr) \ceqq \Br_1$, for $i \in J_1$, using the semipure braid group of Prop.~\ref{prop:semipure_braid_group}.
\end{definition}

Here $\Br_\varphi \sse \Br_{\, \abs{J_p}}$ is the subgroup corresponding to the braiding of pairwise-isomorphic maximal proper subtrees (i.e.,
equal-dimensional eigenspaces of the leading coefficient), acting by conjugation on the cabled braid groups of such subtrees.
Equivalently,
and analogously to Rmk.~\ref{rmk:single_iso_class},
one can embed both $\Br_\varphi$ and $\ms B(t,\bm r_0)$ as subgroups of $\Br_{\, \abs{J_1}}$,
and interpret~\eqref{eq:cabled_braid_group} as an \emph{inner} semidirect product (cf.~also the proof of Thm.~\ref{thm:type_A_local_wmcgs_are_cabled_braid_groups}).

Finally we can prove that~\eqref{eq:cabled_braid_group} is the correct definition, i.e., that indeed this is the group controlling the topology of admissible deformations of type-$A$ irregular classes:

\begin{theorem}
	\label{thm:type_A_local_wmcgs_are_cabled_braid_groups}

	There is a group isomorphism $\pi_1(\bm B_{\ol Q}) \simeq \ms B(T,r)$, and $\ms B(T,r)$ is an extension of $\Aut(T,r)$ by $\ms{P\!B}(T)$.
\end{theorem}

\begin{proof}
	By Prop.~\ref{prop:split_sequence_weyl_fission_type_A} the projection $\bm B_Q \to \bm B_{\ol Q}$ amounts to the Galois covering over the quotient $\bm B_Q \bs \Aut(T,\bm r)$, so the claimed group extension will follow from the first statement---as $\ms{P\!B}(T) \simeq \pi_1(\bm B_Q)$, cf.~\cite{doucot_rembado_tamiozzo_2026_local_wild_mapping_class_groups_and_cabled_braids}.

	The first statement instead can be proven by induction on $p \geq 1$.
	If $p = 1$ then $\ms{P\!B}(T) = \PBr_{\, \abs{J_1}}$, and $\ms B(T,r) \sse \Br_{\, \abs{J_1}}$ is the `semipure' braid group of partition of the leaves into equal-rank nodes: the result then follows from Prop.~\ref{prop:semipure_braid_group}.

	Now suppose $p \geq 2$, and denote by $\bm B(T)$ the space determined by the (unranked) tree as in~\eqref{eq:deformation_space_tree}.
	By definition
	\begin{equation*}
		\bm B(T) = \Conf_{\, \abs{J_p}} \times \prod_{t \in \wt{\mc T}} \bm B(t)^{n(t)} \, ,
	\end{equation*}
	with the usual notation for representatives of maximal proper subtrees, and for the cardinality of their isomorphism classes.
	Then for $t \in \wt{\mc T}$ the base of the wreath product $\on S_{n(t)} \wr \, \Aut(t,\bm r_0)$ acts on the rightmost factor, while $\on S_{n(t)}$ is naturally a subgroup of permutations of the child-nodes of the roots---permuting the (isomorphic) subtrees rooted there.

	It is enlightening to first assume that $\wt{\mc T} = \set{t}$ is a singleton,
	as per Rmk.~\ref{rmk:single_iso_class}.
	Keeping the notation of loc.~cit.,
	we then have $\bm B(T) = \Conf_n \times \bm B(t)^n$.

	Now we have a natural surjective map $\bm B(T) \to \UConf_n$, composing the canonical projection $\bm B(T) \to \Conf_n$ with the Galois covering $\Conf_n \to \UConf_n$; and there is also a Galois covering $\bm B(T) \to \bm B(T,\bm r) \ceqq \bm B(T) \bs \Aut(T,\bm r)$.
	By construction the former factorises through the latter, so there is a commutative triangle of topological spaces:
	\begin{equation*}
		\begin{tikzcd}[column sep=small]
			& \bm B(T) \ar[two heads,swap]{dl}{p} \ar[two heads]{dr}{\pi} & \\
			\bm B(T,\bm r) \ar[two heads,swap]{rr}{\ol \pi} & & \UConf_n
		\end{tikzcd} \, .
	\end{equation*}
	The difference from the case $p = 1$ is that the arrows onto the unordered configuration space are \emph{not} coverings, but rather (locally trivial) fibre bundles with positive-dimensional fibres.
	For $\pi$ this is clear (it is the composition of a trivial bundle and a locally trivial one), while for $\ol{\pi}$ it can be proven as follows.
	If $O \sse \UConf_n$ is an open trivialising set for $\pi$, then $\pi^{-1}(O) = \wt{O} \times \bm B(t)^n$, where $\wt O \simeq \on S_n \times U$ is the preimage of $O$ under the standard Galois covering, and
	\begin{equation*}
		\ol{\pi}^{-1}(O) = p \bigl( \pi^{-1}(O) \bigr) = \bigl( \wt O \times \bm B(t)^n \bigr) \bs \Aut(T,\bm r) \, .
	\end{equation*}
	Now the latter quotient can be taken in two steps: first the action of the base yields
	\begin{equation*}
		\bigl( \wt O \times \bm B(t)^n \bigr) \bs \bigl( \on 1 \wr \, \Aut(t,\bm r_0) \bigr) \simeq \wt O \times \bm B(t,\bm r_0)^n \, ,
	\end{equation*}
	setting analogously $\bm B(t,\bm r_0) \ceqq \bm B(t) \bs \Aut(t,\bm r_0)$.
	In turn,
	the topological subspace $O \times \bm B(t,\bm r_0)^n \sse O \times \on S_n \times \bm B(t,\bm r_0)^n \simeq \wt O \times \bm B(t,\bm r_0)^n$ is a slice for the action of the `complement' subgroup $\on S_n \wr \, 1$.
	In conclusion $\ol{\pi}^{-1}(O) \simeq O \times \bm B(t,\bm r_0)^n$,
	whence a locally trivial fibre bundle
	\begin{equation*}
		\bm B(t,\bm r_0)^n \lhra \bm B(T,\bm r) \lxra{\ol{\pi}} \UConf_n \, .
	\end{equation*}

	Omitting base points,
	this yields the exact group sequence
	\begin{equation*}
		1 \lra \ms B(t,\bm r_0)^n \lra \pi_1\bigl( \bm B(T,\bm r) \bigr) \lxra{\pi_1(\ol{\pi})} \Br_n \lra 1 \, ,
	\end{equation*}
	by the recursive hypothesis, using that $\UConf_n$ is a $K(\pi,1)$-space and that fibres are path-connected.
	Moreover,
	the bundle admits a global section,
	whence a semidirect product decomposition
	\begin{equation*}
		\pi_1 \bigl( \bm B(T,\bm r) \bigr) \simeq \Br_n \wr \, \ms B(t,\bm r_0) \ceqq \Br_n \lts \ms B(t,\bm r_0)^n \, ,
	\end{equation*}
	in accordance with the recursive definition~\eqref{eq:cabled_braid_group}.
	Indeed,
	choose any constant function $f \cl \Conf_n \to \bm B(t)^n$,
	and consider the composition
	\begin{equation}
		\Conf_n \lxra{\Id_{\Conf_n} \ts f} \Conf_n \ts \bm B(t)^n \lthra \bm B(T,\bm r) = \bm B(T) \bs \bigl( \on S_n \wr \Aut(t,\bm r_0) \bigr).
	\end{equation}
	The latter is constant on the $\on S_n$-orbits of $\Conf_n$,
	inducing a continuous section of $\ol\pi$.

	Now the general case where $\wt{\mc T}$ is \emph{not} a singleton essentially follows by taking direct products of groups/topological spaces.
	In some detail,
	we generalise the above argument using the composition of projections
	\begin{equation*}
		\bm B(T) \lra \bm B(T,\bm r) \lra  X_\varphi = \Conf_n \bs \on S_\varphi \, ,
	\end{equation*}
	onto the semiordered configuration space, where $\varphi \cl J_p \to \wt{\mc T}$ is as above.
	Again a 2-step quotient (over any open trivialising subspace $O \sse X_\varphi$) can be taken with respect to the actions of the subgroups
	\begin{equation*}
		\prod_{\wt{\mc T}} \bigl( 1 \wr \, \Aut(t,\bm r_0)^{n(t)} \bigr) \, , \quad \prod_{\wt{\mc T}} \bigl( \on S_{n(t)} \wr \, 1 \bigr) \sse \Aut(T,\bm r) \, ,
	\end{equation*}
	whose (inner) product gives the whole of $\Aut(T,\bm r)$ in view of Lem.~\ref{lem:direct_semidirect_product}.
	The analogous commutative triangle of topological spaces then yields the fibre bundle
	\begin{equation*}
		\prod_{\wt{\mc T}} \bm B(t,\bm r_0)^{n(t)} \lhra \bm B(T,\bm r) \lxra{\ol{\pi}} X_\varphi \, .
	\end{equation*}
	In conclusion there is an exact group sequence
	\begin{equation*}
		1 \lra \prod_{\wt{\mc T}} \ms B(t,\bm r_0)^{n(t)} \lra \pi_1\bigl(\bm B(T,\bm r)\bigr) \lxra{\pi_1(\ol{\pi})} \Br_\varphi \lra 1 \, ,
	\end{equation*}
	using: (i) the recursive hypothesis; (ii) Prop.~\ref{prop:semipure_braid_group}; (iii) the fact that fibres are path-connected; and (iv) the fact that $X_\varphi$ is a covering of a $K(\pi,1)$-space---%
	and so it is also $K(\pi,1)$.
	Again this admits global sections via continuous maps $X_\varphi \to \prod_{\wt{\mc T}} \bm B(t)^{n(t)}$.
\end{proof}

\begin{remark}
	The auxiliary fibre bundle
	\begin{equation*}
		\on S_n \times \bm B(t)^n \lhra \bm B(T) \lxra{\pi} \UConf_n \, ,
	\end{equation*}
	which appears in the above proof when $\wt{\mc T} = \set{t}$, yields the following exact group sequence:
	\begin{equation*}
		1 \lra \ms{P\!B}(t)^n \lra \ms{P\!B}(T) \lxra{\pi_1(\pi)} \Br_n \lra \on S_n \lra 1 \, .
	\end{equation*}
	This is recovering the fact that $\ms{P\!B}(T) \bs \ms{P\!B}(t)^n \simeq \PBr_n$, considering the direct product over the nodes of the (unranked) fission tree.
\end{remark}

\section*{Outlook}

There is a `twisted' version of irregular types/classes~\cite{boalch_yamakawa_twisted_wild_character_varieties,boalch_yamakawa_2020_diagrams_for_nonabelian_hodge_spaces_on_the_affine_line,doucot_diagrams_and_irregular_connections_on_the_riemann_sphere}; and there exist `global' deformations of (twisted, bare) wild Riemann surfaces, defined as in~\cite{doucot_rembado_tamiozzo_2026_local_wild_mapping_class_groups_and_cabled_braids}: we plan to consider these elsewhere.\fn{
	The twisted local case, both pure and full/nonpure, in type $A$, has now been studied in~\cite{boalch_doucot_rembado_2025_twisted_local_wild_mapping_class_groups_configuration_spaces_fission_trees_and_complex_braids};
	and the general case is in~\cite{doucot_rembado_yamakawa_twisted_local_g_wild_mapping_class_groups}.
	The global case instead has been addressed both in~\cite{boalch_doucot_rembado_2025_twisted_local_wild_mapping_class_groups_configuration_spaces_fission_trees_and_complex_braids} and~\cite{doucot_rembado_tamiozzo_moduli_spaces_of_untwisted_wild_riemann_surfaces}.
}

A further question is to obtain explicit expressions for the actions of the wild mapping class groups on wild character varieties, laying the ground to investigate the resulting Poisson/symplectic dynamics.
E.g. it is known that the algebraic solutions of Painlevé VI correspond to monodromy representations with finite orbits under the braid group action~\cite{dubrovin_mazzocco_2000_monodromy_of_certain_painleve_vi_transcendents_and_reflection_groups,boalch_2005_from_klein_to_painleve_via_fourier_laplace_and_jimbo}, so it is natural to ask whether a similar phenomenon holds in the wild case.

Finally recall~\cite{doucot_rembado_tamiozzo_2026_local_wild_mapping_class_groups_and_cabled_braids} constructed a fine moduli scheme $\ms{I\!\!T}^{\leq p}_{\bm d}$ of irregular types of bounded pole order $p \in \mb Z_{\geq 1}$, and with given pole order $d_{\alpha} \in \set{0,\dc,p}$ after evaluation at each root $\alpha \in \Phi_{\mf g}$---for any complex reductive group $G$.
In particular the $\mb C$-points $\ms{I\!\!T}^{\leq p}_{\bm d}(\Spec \mb C)$ recover the above `universal' deformation space $\bm B_Q$, for any irregular type $Q \in \mf t \ots \ms T^{\leq p}_{\Sigma,a}$ with $\ord(q_{\alpha}) = d_{\alpha}$.
We plan to construct fine moduli spaces of irregular classes, also in the twisted setting, presumably via GIT-theoretic quotients of the moduli spaces of irregular types.

\section*{Acknowledgements}

We thank R. Ontani for a helpful discussion about fibre bundles, and M. Tamiozzo for sharing his precious insight about the algebraic structure of the deformation spaces; and we thank P. Boalch for his past guidance.

\appendix

\section{Basic notions/notations}
\label{sec:app_notions}

\subsection*{Permutations and partitions}

For an integer $n \geq 0$ we write $\ul n \ceqq \set{1,\dotsc,n}$ (so $\ul 0 = \varnothing$), and let $\on S_n$ be the symmetric group of permutations of $\ul n$ (so $\on S_0$ and $\on S_1$ are trivial).
The action of a permutation $\sigma \in \on S_n$ is denoted by $j \mapsto \sigma_j \in \ul n$, and we compose them right-to-left.
More generally if $I$ is a finite set we denote by $\on S_I$ its symmetric group of permutations, so $\on S_I \simeq \on S_{\, \abs I}$ by choosing a total order on $I$---%
which is of course broken by the group action.

If $I$ and $J$ are sets, a $J$-\emph{partition} of $I$ is a surjection $\phi \cl I \thra J$, which is equivalent to giving a disjoint union
\begin{equation*}
	I = \coprod_{j \in J} I_j \, , \qquad I_j \ceqq \phi^{-1}(j) \sse I \, ,
\end{equation*}
with (nonempty) parts indexed by $J$.

If $I$ is finite, and all parts have the same cardinality $m = \abs{I_j} \geq 1$, then we can consider the subgroup permuting all parts, and further the elements within each part: this is the (restricted) \emph{wreath} product
\begin{equation*}
	\on S_J \wr \, \on S_m = \on S_J \wr_{J} \, \on S_m \, ,
\end{equation*}
using the natural $\on S_J$-action on $J$.
Namely,
one sets
\begin{equation*}
	\on S_J \wr \, \on S_m \ceqq \on S_J \lts (\on S_m)^{\abs J} \, ,
\end{equation*}
with respect to the action of $\on S_J$ given by
\begin{equation*}
	\sigma \cdot \bm \tau = \prod_{j \in J} \tau^{(\sigma^{-1}_j)} \, , \qquad \sigma \in \on S_j \, , \quad \bm \tau = \prod_{j \in J} \tau^{(j)} \in (\on S_m)^{\abs{J}} \, ,
\end{equation*}
where $\tau^{(j)} \in \on S_{I_j} \simeq \on S_m$ for $j \in J$.
Elements of $\on S_J \wr \, \on S_m$ are then written $(\sigma;\bm \tau)$.

\subsection*{Weyl actions}

Let $(V,\Phi)$ be a root system in the finite-dimensional complex vector space $V$, and $W = W(\Phi) \sse \GL_{\mb C}(V)$ the Weyl group.
If $S \sse V$ is a subset, its \emph{setwise} Weyl-stabiliser is the subgroup
\begin{equation*}
	\Stab_W(S) = \Set{ w \in W | w(S) \sse S } \sse W \, ,
\end{equation*}
and its \emph{pointwise} Weyl-stabiliser is the subgroup
\begin{equation*}
	W_S = \Set{ w \in W | S \sse \Ker(w - \Id_V) } \sse \Stab_W(S) \, .
\end{equation*}
The latter is a \emph{parabolic} subgroup of $W$---thinking of the Weyl group as a reflection group---,
and it is a normal subgroup of the former.
Clearly $W_S = W_{\spann_{\mb C}(S)}$, and if $\spann_{\mb C}(S_1) \sse \spann_{\mb C}(S_2)$ then $W_{S_1} \sse W_{S_2}$;
furthermore,
if $U_1, U_2 \sse V$ are subspaces then
\begin{equation*}
	W_{U_1 + U_2} = W_{U_1} \cap W_{U_2} \sse W \, .
\end{equation*}
On the other hand, tautologically, if $W_i = \Stab_W (U_i)$ then
\begin{equation*}
	\Stab_{W_1}(U_2) = W_1 \cap \Stab_W(U_2) = W_2 \cap \Stab_W(U_1) = \Stab_{W_2}(U_1) \, .
\end{equation*}

We also identify $W$ with the Weyl group $W(\Phi^{\dual}) \sse \GL_{\mb C}(V^{\dual})$ of the dual/inverse root system, via $w \mapsto \prescript{t}{}w^{-1}$,
cf.~\cite[Ch.~VI, \S~1.1]{bourbaki_1968_groupes_et_algebres_de_lie_chapitres_iv_v_vi}.

\subsection*{Braid groups}

For an integer $n \geq 0$ we denote by $\PBr_n$ the \emph{pure braid group} on $n$ strands---so $\PBr_0$ and $\PBr_1$ are trivial.
It is the fundamental group of the space
\begin{equation}
	\label{eq:type_A_complement}
	\Conf_n = \Conf_n(\mb C) \ceqq \mb C^n \sm \bigcup_{1 \leq i \neq j \leq n} H_{ij} \, ,
\end{equation}
where
\begin{equation*}
	H_{ij} = \Set{ (z_1,\dc,z_n) \in \mb C^n \mid z_i = z_j } \sse \mb C^n \, .
\end{equation*}
In particular $\Conf_1 = \mb C$, and in general this yields the space of \emph{ordered} configurations of $n$ points in the complex plane.
The symmetric group $\on S_n$ acts naturally on~\eqref{eq:type_A_complement}, and the projection
\begin{equation}
	\label{eq:unordered_configuration_space}
	\Conf_n \lra \UConf_n \ceqq \Conf_n \bs \on{S}_n \, ,
\end{equation}
to the space of \emph{unordered} configurations, is a Galois covering.
The (full/nonpure) \emph{braid group} is $\Br_n = \pi_1(\UConf_n)$, and the associated exact group sequence
\begin{equation}
	\label{eq:exact_sequence_braid_group}
	1 \lra \PBr_n \lra \Br_n \lxra{p_n} S_n \lra 1
\end{equation}
corresponds to the braid group `augmentation', i.e., the morphism taking the permutation underlying the braiding of the $n$ strands~\cite{artin_1947_theory_of_braids}.

More generally, for a split reductive Lie algebra $(\mf g,\mf t)$ we consider the root-hyperplane complement
\begin{equation*}
	\mf t_{\reg} = \mf t \setminus \bigcup_{\alpha \in \Phi_\mf g} \Ker(\alpha) \sse \mf t \, , \qquad \Phi_\mf g = \Phi(\mf g,\mf t) \, ,
\end{equation*}
generalising~\eqref{eq:type_A_complement} in type $A$, and $\PBr_\mf g = \pi_1(\mf t_{\reg})$ is the \emph{pure} $\mf g$-\emph{braid group}, a.k.a.~the generalised (Artin--Tits) braid group of type $\mf g$~\cite{brieskorn_1971_die_fundamentalgruppe_des_raumes_der_regulaeren_orbits_einer_endlichen_komplexen_spiegelungsgruppe,brieskorn_saito_1972_artin_gruppen_und_coxeter_gruppen,deligne_1972_les_immeubles_des_groupes_de_tresses_generalises,brieskorn_1973_sur_les_groupes_de_tresses_d_apres_v_i_arnold}.
The Weyl group $W_{\mf g} = W(\Phi_\mf g)$ acts freely on $\mf t_{\reg}$, and $\mf t_{\reg} \to \ol{\mf t}_{\reg} \ceqq \mf t_{\reg} \bs W_{\mf g}$ is a Galois covering.
Then $\Br_\mf g \ceqq \pi_1 (\ol{\mf t}_{\reg})$ is the \emph{full/nonpure} $\mf g$-braid group, and there is an exact group sequence generalising~\eqref{eq:exact_sequence_braid_group}:
\begin{equation*}
	1 \lra \PBr_\mf g \lra \Br_\mf g \lra W_\mf g \lra 1 \, .
\end{equation*}

\subsection*{Trees}

All trees in this article are finite,
rooted,
and have levels (cf.~\cite{boalch_2025_counting_the_fission_trees_and_nonabelian_hodge_graphs_untwisted_case}).
Concretely,
a \emph{fission tree} $T = (T_0,\bm \phi)$ of height $p \geq 1$ is the data of:
(i) a finite set $T_0$,
with a partition $T_0 = \coprod_{l = 1}^{p+1} J_l$---%
with nonempty parts---%
such $J_{p+1} = \set{\ast}$ is a singleton;
and (ii) a function $\bm \phi \cl T_0 \sm \set{\ast} \to T_0$ such that $\bm \phi(J_l) = J_{l+1}$ for $l \in \set{1,\dc,p}$.
(This yields a $J_l$-partition $\phi_l \ceqq \eval[1]{\bm\phi}_{J_l} \cl J_l \thra J_{l+1}$.)

The elements of $T_0$ are the \emph{nodes} of the tree, and $\bm \phi(i)$ is the \emph{parent-node} of $i \in T_0 \sm \set{\ast}$---so $\ast \in J_{p+1}$ is the \emph{root}, while $J_1 \sse T_0$ contains the \emph{leaves}.
Conversely $\bm \phi^{-1}(i) \sse T_0$ is the set of \emph{child-nodes} of $i \in T_0$,
so that in particular $J_p = \phi^{-1}(\ast)$ is the set of child-nodes of the root,
on which the maximal proper subtrees of $T$ are rooted.

\section{Deferred proofs}
\label{sec:app_lemmata}

\begin{proof}[Proof of Lem.~\ref{lem:recursive_stabiliser_equals_flag_stabiliser}]
	\label{proof:lem_recursive_stabiliser_equals_flag_stabiliser}

	We can recursively prove that
	\begin{equation*}
		W_i = \bigcap_{i \leq j \leq p} \Stab_{W_\mf g}(U_j) \, , \qquad i \in \set{1,\dc,p} \, .
	\end{equation*}

	The base $i = p$ is tautological, and then
	\begin{equation*}
		W_{i-1} = \Set{ w \in W_i | w(U_{i-1}) \sse U_{i-1} } = \Set{ w \in W_{\mf g} | w(U_j) \sse U_j \text{ for } j \geq i-1 } \, ,
	\end{equation*}
	using~\eqref{eq:recursive_stabiliser} and the recursive hypothesis.
\end{proof}

\begin{proof}[Proof of Lem.~\ref{lem:wreath_quotients}]
	\label{proof:lem_wreath_quotients}

	By definition $\on S_m \wr \, P = \on S_m \lts P^m$, with respect to the natural permutation action $\on S_m \to \Aut(P^m)$.
	Then $1 \wr N = 1 \lts N^m \sse \on S_m \lts P^m$, and it is a normal subgroup since it is normalised by $1 \wr P =1 \lts  P^m$ and stabilised by the permutation action.
	Hence the quotient on the left-hand of~\eqref{eq:wreath_quotients} is well defined.

	Now there is an induced action $\on S_m \to \Aut(Q^m)$, where $Q \ceqq P \bs N$, and finally the natural surjective group morphism $\on S_m \wr \, P \to \on S_m \wr \, Q$ vanishes on $1 \wr N$.
\end{proof}

\begin{proof}[Proof of Lem.~\ref{lem:semidirect_product_symmetric}]
	\label{proof:lem_semidirect_product_symmetric}

	The compatibility with the product follows from~\eqref{eq:permutation_conjugation_action_symmetric_group} (which in turn is equivalent to the action-operad axiom for $\ms S$~\cite[Eq.~4.1.2]{yau_2021_infinity_operads_and_monoidal_categories_with_group_equivariance}), and from the fact that the block permutation operation is a group morphism in this case.

	Injectivity follows from the identity
	\begin{equation*}
		\on S_n \braket{ \bm k } \cap (\on S_k)^n = 1 \sse \on S_{nk} \, ,
	\end{equation*}
	where $\on S_n \braket{ \bm k } \sse \on S_{nk}$ is the image of the block permutation operation $\on S_n \to \on S_{nk}$.
\end{proof}

\begin{proof}[Proof of Lem.~\ref{lem:direct_semidirect_product}]
	\label{proof:lem_direct_semidirect_product}

	There is a natural bijection
	\begin{equation*}
		\prod_I (p_i,n_i) \lmt \Bigl( \prod_I p_i,\prod_I n_i \Bigr) \, ,
	\end{equation*}
	between the underlying sets, and one can show it is compatible with the semidirect multiplication.

	Indeed choose elements $p'_i,p_i \in P_i$ and $n'_i,n_i \in N_i$ for $i \in I$, so that
	\begin{align*}
		\prod_I (p'_i,n'_i) \prod_I (p_i,n_i) & = \prod_I (p'_i,n'_i) \bullet_i (p_i,n_i)                            \\
		                                      & = \prod_I (p'_ip_i,\rho_i(p_i)n'_in_i) \in \prod_I P_i \lts N_i \, ,
	\end{align*}
	which is mapped to $\Bigl( \prod_I p'_ip_i, \prod_I \rho_i(p_i)n'_in_i \Bigr) \in P \lts N$.
	Conversely
	\begin{equation*}
		\Bigl(\prod_I p'_i,\prod_I n'_i \Bigr) \bullet \Bigl( \prod_I p_i,n_i \Bigr) = \Bigl( \prod_I p'_i\prod_I p_i, \rho \bigl( \prod_I p_i \bigr)\prod_I n'_i \prod_I n_i \Bigr) \in P \lts N \, ,
	\end{equation*}
	which coincides with the above---using the product action and the direct product multiplication.
\end{proof}

\section{Relations to isomonodromy systems}
\label{sec:app_isomonodromy_systems}

On the other side of the Riemann--Hilbert--Birkhoff correspondence there is a Poisson fibre bundle analogous to~\eqref{eq:betti_bundle}, viz.
\begin{equation}
	\label{eq:de_rham_bundle}
	\ul{\mc M}_{\dR} \lxra{\pi} \bm B \, ,
\end{equation}
whose fibres (the de Rham spaces) are moduli spaces of irregular singular connections on principal $G$-bundles.
This is equipped with the pullback (flat, nonlinear) isomonodromy connection; see in particular~\cite[Fig.~1]{boalch_2001_symplectic_manifolds_and_isomonodromic_deformations}, which spells out the picture of the wild nonabelian Gauß--Manin connection on the Betti side (extending the tame case~\cite{simpson_1994_moduli_of_representations_of_the_fundamental_group_of_a_smooth_projective_variety_i,simpson_1994_moduli_of_representations_of_the_fundamental_group_of_a_smooth_projective_variety_ii}).

Now one can choose a local trivialisation of~\eqref{eq:de_rham_bundle}, i.e., an isomorphism of fibre bundles
\begin{equation}
	\label{eq:local_de_rham_trivialisation}
	\begin{tikzcd}[column sep=small]
		\eval[1]{\ul{\mc M}_{\dR}}_O \ceqq \pi^{-1}(O) \ar{rr}{\simeq} \ar[swap]{dr}{\pi} & & O \times M \ar{dl}{p_1} \\
		& O &
	\end{tikzcd} \, ,
\end{equation}
over an open subspace $O \sse \bm B$, for a \emph{fixed} Poisson manifold $\bigl( M,\set{ \cdot,\cdot } \bigr)$.
Then the isomonodromy connection (on the upper-left corner of~\eqref{eq:local_de_rham_trivialisation}) can be given by explicit nonlinear first-order partial differential equations in local coordinates $\bm t = (t_1,\dc,t_d)$ on $O$, where $d = \dim (\bm B)$, for local sections over the trivialising locus.

Moreover the difference between the isomonodromy connection  and the trivial connection (on the upper-right corner of~\eqref{eq:local_de_rham_trivialisation}) can be `integrated'\fn{
	The difference of the corresponding horizontal distributions---inside $TO \times TM \xrightarrow{Tp_1} TO$---is given by $O$-dependent (vertical) vector fields $X_i \cl O \times M \to TM$ on the fibre, and one has $\dif H_i = \Braket{\omega,X_i}$ (cf.~\cite[\S~5]{boalch_2012_simply_laced_isomonodromy_systems}).}
to a nonautonomous Hamiltonian system
\begin{equation}
	\label{eq:isomonodromy_system}
	\bm H = (H_1,\dc,H_d) \cl M \times O \lra \mb C^d \, .
\end{equation}
In this Hamiltonian viewpoint, the symplectic nature of isomonodromic deformations is equivalent to the integrability of~\eqref{eq:isomonodromy_system}.
The latter amounts to the identities
\begin{equation*}
	\Set{ H_i,H_j } + \pd{H_i}{t_j} - \pd{H_j}{t_i} = 0 \, , \qquad i,j \in \set{1,\dc,d} \, .
\end{equation*}

Hence the local coordinates become times of isomonodromic deformations over $O \sse \bm B$: but in principle they are \emph{not} intrinsically associated with isomonodromic deformations, contrary to the flat Ehresmann connections on~\eqref{eq:betti_bundle} and~\eqref{eq:de_rham_bundle} (one needs a choice of `initial' trivialisation~\cite[Rmk.~7.1]{boalch_2001_symplectic_manifolds_and_isomonodromic_deformations}; cf.~\cite{yamakawa_2019_fundamental_two_forms_for_isomonodromic_deformations}).

\vspace{5pt}

Examples of such isomonodromy systems abound, with far-reaching applications already in the genus-zero case, famously encompassing (generalisations of) the Painlevé equations~\cite{okamoto_1987_studies_on_the_painleve_equations_i_sixth_painleve_equation_p_vi,okamoto_1987_studies_on_the_painleve_equations_ii_fifth_painleve_equation_p_v,okamoto_1986_studies_on_the_painleve_equations_iii_second_and_fourth_painleve_equations_p_ii_and_p_iv,bertola_cafasso_rubtsov_2018_noncommutative_painleve_equations_and_systems_of_calogero_type,cafasso_gavrylenko_lisovyy_2019_tau_functions_as_widom_constants} and the Schlesinger system~\cite{schlesinger_1905_ueber_die_loesungen_gewisser_linearer_differentialgleichungen_als_funktionen_der_singulaeren_punkte} (cf.~also~\cite{miwa_1981_painleve_property_of_monodromy_preserving_deformation_equations_and_the_analyticity_of_tau_functions}).
The (Harnad-)dual version of the Schlesinger system, on the other side of the Fourier--Laplace transform~\cite{harnad_1994_dual_isomonodromic_deformations_and_moment_maps_to_loop_algebras,yamakawa_2016_fourier_laplace_transform_and_isomonodromic_deformations}, was considered in~\cite{boalch_2002_g_bundles_isomonodromy_and_quantum_weyl_groups}, and the combination of Schlesinger and its dual yield the system of Jimbo--Miwa--Môri--Sato (JMMS)~\cite{jimbo_miwa_mori_sato_1980_density_matrix_of_an_impenetrable_bose_gas_and_the_fifth_painleve_transcendent}.
Recall the Painlevé property of the JMMS equations was studied in~\cite{miwa_1981_painleve_property_of_monodromy_preserving_deformation_equations_and_the_analyticity_of_tau_functions}.

Note also that~\cite{harnad_1994_dual_isomonodromic_deformations_and_moment_maps_to_loop_algebras} links previous papers~\cite{adams_harnad_previato_1988_isospectral_hamiltonian_flows_in_finite_and_infinite_dimensions_i_generalized_moser_systems_and_moment_maps_into_loop_algebras,adams_harnad_hurtubise_1990_isospectral_hamiltonian_flows_in_finite_and_infinite_dimensions_ii_integration_of_flows} about \emph{isospectral} deformations to the \emph{isomonodromic} deformations of JMMS.
However this hardly the whole story: a rigorous treatment of the degeneration of the isomonodromic deformations of JMMS into a combination of the isospectral deformations and the Whitham dynamics is still open, cf.~\cite{takasaki_1998_dual_isomonodromic_problems_and_whitham_equations} for a related conjecture, and~\cite{xu_representations_of_quantum_groups_arising_from_the_stokes_phenomenon} in the quantum case---with applications to quantum groups and canonical bases.
In particular the `cactus' groups naturally appear: they are fundamental groups of certain compactifications of the real points in the base space, controlling the asymptotics zones of the nonlinear isomonodromy equations,
and are analogues of the braid groups encountered here.

Finally a generalisation of all the above was derived in~\cite{boalch_2012_simply_laced_isomonodromy_systems}.
This latter setup brings about nongeneric isomonodromic deformations, considering connections with several levels, which extend examples of the seminal paper~\cite{jimbo_miwa_ueno_1981_monodromy_preserving_deformation_of_linear_ordinary_differential_equations_with_rational_coefficients_i_general_theory_and_tau_function}.
(Recall the set of levels is the set of nonzero pole orders of the irregular types, evaluated at each root.)
Importantly this is more symmetric than op.~cit., which in turn is one of our main motivations for studying the `deeper' nongeneric case: in particular in~\cite{boalch_2012_simply_laced_isomonodromy_systems} the group $\SL_2(\mb C)$ acts on the bundle of de Rham spaces via automorphisms of the 1-dimensional Weyl algebra, and contains the Fourier--Laplace transform as the element
$\begin{pmatrix}
		0 & -1 \\
		1 & 0
	\end{pmatrix}$ (cf.~\cite{doucot_hohl_2025_a_topological_algorithm_for_the_fourier_transform_of_stokes_data_at_infinity}).

\vspace{5pt}

Importantly all these Hamiltonian systems have numerous applications in mathematical physics, notably in integrable hierarchies of differential equations such as KdV~\cite{gerard_1983_la_geometrie_des_transcendantes_de_p_painleve}, and in 2d conformal field theory after quantisation (e.g.~\cite{nagoya_sun_2010_confluent_primary_fields_in_the_conformal_field_theory,yamada_2011_a_quantum_isomonodromy_equation_and_its_application_to_n_equal_2_su_n_gauge_theories,alday_gaiotto_tachikawa_2010_liouville_correlation_functions_from_four_dimensional_gauge_theories,gaiotto_2013_asymptotically_free_n_equal_2_theories_and_irregular_conformal_blocks,felder_rembado_2023_singular_modules_for_affine_lie_algebras_and_applications_to_irregular_wznw_conformal_blocks,calaque_felder_rembado_wentworth_wild_orbits_and_generalised_singularity_modules_stratifications_and_quantisation}, opening to `irregular' conformal blocks and the AGT correspondence).
More precisely the quantisation of the Schlesinger system leads to the Knizhnik--Zamolodchikov connection (KZ)~\cite{reshetikhin_1992_the_knizhnik_zamolodchikov_system_as_a_deformation_of_the_isomonodromy_problem,harnad_1996_quantum_isomonodromic_deformations_and_the_knizhnik_zamolodchikov_equations}, while the quantisation of the dual Schlesinger system (as in~\cite{boalch_2002_g_bundles_isomonodromy_and_quantum_weyl_groups}) leads to the Casimir connection of de Concini/Millson--Toledano Laredo~\cite{millson_toledanolaredo_2005_casimir_operators_and_monodromy_representations_of_generalised_braid_groups}.
The quantisation of the JMMS systems, generalising the above and recovering the connection of Felder--Markov--Tarasov--Varchenko~\cite{felder_markov_tarasov_varchenko_2000_differential_equations_compatible_with_kz_equations}, was considered in~\cite{rembado_2019_simply_laced_quantum_connections_generalising_kz}; further op.~cit. constructed a quantisation of the more general `simply-laced' systems of~\cite{boalch_2012_simply_laced_isomonodromy_systems} (this is resumed in the table in the introduction of~\cite{rembado_2019_simply_laced_quantum_connections_generalising_kz}); cf.~\cite{yamakawa_2022_quantization_of_simply_laced_isomonodromy_systems_by_the_quantum_spectral_curve_method} for a different construction of `quantum' simply-laced isomonodromy systems, and~\cite{nagoya_sun_2010_confluent_primary_fields_in_the_conformal_field_theory,nagoya_sun_2011_confluent_kz_equations_for_sl_n_with_poincare_rank_2_at_infinity,gaiur_mazzocco_rubtsov_2023_isomonodromic_deformations_confluence_reduction_and_quantisation} for a `confluence' viewpoint on the quantisation of irregular singularities,
and~\cite{chaffe_rembado_yamakawa_wild_genus_zero_quantum_de_rham_spaces} for a deformation-quantisation of nonresonant wild de Rham spaces in the general nongeneric untwisted case---%
for any polar divisor and complex reductive structure group.

\vspace{5pt}

Finally,
encoding the irregular moduli in the base curve, and constructing bundles over their (admissible) deformations, is also helpful for the quantisation of the extended `classical' symmetries of isomonodromy systems.
In particular the quantised $\SL_2(\mb C)$-symmetries~\cite{rembado_2020_symmetries_of_the_simply_laced_quantum_connections_and_quantisation_of_quiver_varieties} generalise the Howe duality~\cite{baumann_1999_the_q_weyl_group_of_a_q_schur_algebra}, which in turn was used in~\cite{toledanolaredo_2002_a_kohno_drinfeld_theorem_for_quantum_weyl_groups} to compute the monodromy of the Casimir connection in terms of that of KZ (cf.~also~\cite{tarasov_varchenko_2002_duality_for_knizhnik_zamolodchikov_and_dynamical_equations}): this latter example of generic `quantum' monodromy action brings about the $G$-braid groups which we generalise in this series of papers.

\section{\for{toc}{List of notation}\except{toc}{List of some nonstandard notation (in rough order of appearance)}}
\label{sec:app_notation_list}

\begin{longtable}{ p{.16\textwidth}  p{.84\textwidth} }
	$\Sigma$                       & Riemann surface                                                                             \\
	$G$                            & connected complex reductive Lie group                                                       \\
	$\ul{\mc M}_{\on B}$           & Poisson/symplectic fibration of Betti spaces                                                \\
	$\mf g$                        & Lie algebra of $G$                                                                          \\
	$\mf t$                        & Cartan subalgebra of $\mf g$                                                                \\
	$T$                            & maximal (algebraic) torus in $G$                                                            \\
	$Q$                            & irregular type                                                                              \\
	$a$                            & point of $\Sigma$                                                                           \\
	$A_i$                          & coefficients of $Q$                                                                         \\
	$W_{\mf g}$                    & Weyl group of $(\mf g,\mf t)$                                                               \\
	$\ol Q$                        & irregular class underlying $Q$                                                              \\
	$\bm \Sigma$                   & wild Riemann surface                                                                        \\
	$\bm B_Q$                      & space of admissible deformations of $Q$                                                     \\
	$\bm B_{A_i}$                  & space of admissible deformations of $A_i$                                                   \\
	$\bm B_{\ol Q}$                & space of admissible deformations of $\ol Q$                                                 \\
	$W_{\mf g \mid \bm{\mf h}}$    & subquotient of $W_{\mf g}$ acting freely on $\bm B_Q$                                       \\
	$(T,\bm r)$                    & ranked fission tree                                                                         \\
	$\Aut(T,\bm r)$                & automorphisms of $(T,\bm r)$                                                                \\
	$\ms B(T,\bm r)$               & full/nonpure cabled braid group of $(T,\bm r)$                                              \\
	$\Phi_{\mf g}$                 & root system of $(\mf g,\mf t)$                                                              \\[3pt]
	$\wh{\ms O}_{\Sigma,a}$        & completed local ring of $\Sigma$ at $a$                                                     \\[3pt]
	$\wh{\ms K}_{\Sigma,a}$        & fraction field of $\wh{\ms O}_{\Sigma,a}$                                                   \\[3pt]
	$\ms T_{\Sigma,a}$             & quotient of $\wh{\ms K}_{\Sigma,a}$ modulo $\wh{\ms O}_{\Sigma,a}$                          \\
	$d_{\alpha}$                   & pole order of $q_{\alpha} = (\alpha \ots 1)Q$                                               \\
	$\bm d$                        & tuple of the $d_{\alpha}$                                                                   \\
	$\Gamma_Q$                     & pure local WMCG                                                                             \\
	$\Gamma_{\ol Q}$               & full/nonpure local WMCG                                                                     \\
	$\mf h_i$                      & nested stabilisers of $A_1,\dc,A_p$                                                         \\
	$H_i$                          & connected subgroup of $G$ with Lie algebra $\mf h_i$                                        \\
	$\Phi_{\mf h_i}$               & Levi subsystem of $\Phi_{\mf g}$                                                            \\
	$W_{\mf h_i}$                  & Weyl group of $(\mf h_i,\mf t)$                                                             \\
	$U_i$                          & intersection of the root hyperplanes of $\Phi_{\mf h_i}$                                    \\
	$W_i$                          & nested setwise stabilisers of the $U_i$                                                     \\
	$\Stab_{W_i}(\bm B_{A_i})$     & setwise stabiliser of $\bm B_{A_i}$ in $W_i$                                                \\
	$\bm U$                        & flag of the subspaces $U_i$ in $\mf t$                                                      \\
	$(W_{\mf g})_{U_i}$            & pointwise stabiliser of $U_i$ in $W_{\mf g}$                                                \\
	$\PBr_n$                       & pure braid group on $n$ strands                                                             \\
	$\Br_n$                        & full/nonpure braid group on $n$ strands                                                     \\
	$\Dih_n$                       & dihedral group of order $2n$                                                                \\
	$\Delta_{\mf g}$               & base of simple roots for $\Phi_{\mf g}$                                                     \\
	$\Conf_n$                      & configuration space of $n$ ordered points in $\mb C$                                        \\
	$\UConf_n$                     & configuration space of $n$ unordered points in $\mb C$                                      \\
	$\ul n$                        & the set $\set{1,\dc,n}$                                                                     \\
	$\on S_I$                      & group of permutations of a set $I$                                                          \\
	$I_i$                          & parts of $\ul n$ defined by a root subsystem of $A_{n-1}$                                   \\
	$J$                            & index set for the parts $I_i$                                                               \\
	$K_i$                          & collection of parts $I_j$ with $i \geq 0$ elements                                          \\[3pt]
	$\on S_\varphi$                & group of permutations preserving a partition $\varphi$                                      \\
	$p_n$                          & augmentation group morphism of $\Br_n$                                                      \\
	$\Br_\varphi$                  & group of braids with underlying permutations in $\on S_\varphi$                             \\
	$T$                            & fission tree                                                                                \\
	$J_i$                          & levels of $T$                                                                               \\
	$T_0$                          & nodes of $T$                                                                                \\
	$\bm \phi$                     & parent-node function of $T$                                                                 \\
	$\bm r$                        & rank function of $T$                                                                        \\
	$\mc T$                        & set of maximal proper subtrees of $T$                                                       \\
	$t$                            & an element of $\mc T$                                                                       \\
	$\Aut(T,\bm r)$                & automorphism group of $(T,\bm r)$                                                           \\
	$t_0$                          & nodes of $t$                                                                                \\
	$\bm r_0$                      & restriction of $\bm r$ to $t_0$                                                             \\
	$\ms S$                        & symmetric braid group operad                                                                \\
	$\gamma^{\ms S}$               & composition of $\ms S$                                                                      \\
	$\ms B$                        & full/nonpure braid group operad                                                             \\
	$\gamma^{\ms B}$               & composition of $\ms B$                                                                      \\
	$\ms{P\!B}(T)$                 & pure cabled braid group of $T$                                                              \\
	$\ms{I\!\!T}^{\leq p}_{\bm d}$ & moduli scheme of irregular types of bounded pole order, and given pole order along any root \\
	$\ul{\mc M}_{\dR}$             & Poisson/symplectic fibration of de Rham spaces                                              \\
\end{longtable}

\bibliographystyle{amsalpha}
\bibliography{/home/gabriele/Desktop/bibliography_macros/bibliography}
\end{document}